\documentclass[a4paper,11pt,reqno,abstract=on]{amsart}

\usepackage[french,english]{babel}

\usepackage[T1]{fontenc}

\usepackage{dsfont,amssymb}

\usepackage{color}

\usepackage{hyperref}

\usepackage[nobysame,alphabetic,initials,msc-links]{amsrefs}

\DeclareMathAlphabet{\mathpzc}{OT1}{pzc}{m}{it}

\DefineSimpleKey{bib}{how}
\renewcommand{\PrintDOI}[1]{%
  \href{http://dx.doi.org/#1}{{\tt DOI:#1}}%
}
\renewcommand{\eprint}[1]{#1}
\BibSpec{book}{%
    +{}  {\PrintPrimary}                {transition}
    +{.} { \PrintDate}                  {date}
    +{.} { \textit}                     {title}
    +{.} { }                            {part}
    +{:} { \textit}                     {subtitle}
    +{,} { \PrintEdition}               {edition}
    +{}  { \PrintEditorsB}              {editor}
    +{,} { \PrintTranslatorsC}          {translator}
    +{,} { \PrintContributions}         {contribution}
    +{,} { }                            {series}
    +{,} { \voltext}                    {volume}
    +{,} { }                            {publisher}
    +{,} { }                            {organization}
    +{,} { }                            {address}
    +{,} { }                            {status}
    +{,} { \PrintDOI}                   {doi}
    +{,} { \PrintISBNs}                 {isbn}
    +{}  { \parenthesize}               {language}
    +{}  { \PrintTranslation}           {translation}
    +{;} { \PrintReprint}               {reprint}
    +{.} { }                            {note}
    +{.} {}                             {transition}
    +{}  {\SentenceSpace \PrintReviews} {review}
}
\BibSpec{article}{%
    +{}  {\PrintAuthors}                {author}
    +{,} { \textit}                     {title}
    +{.} { }                            {part}
    +{:} { \textit}                     {subtitle}
    +{,} { \PrintContributions}         {contribution}
    +{.} { \PrintPartials}              {partial}
    +{,} { }                            {journal}
    +{}  { \textbf}                     {volume}
    +{}  { \PrintDatePV}                {date}
    +{,} { \issuetext}                  {number}
    +{,} { \eprintpages}                {pages}
    +{,} { }                            {status}
    +{,} { \PrintDOI}                   {doi}
    +{,} { \eprint}        {eprint}
    +{}  { \parenthesize}               {language}
    +{}  { \PrintTranslation}           {translation}
    +{;} { \PrintReprint}               {reprint}
    +{.} { }                            {note}
    +{.} {}                             {transition}
    +{}  {\SentenceSpace \PrintReviews} {review}
}
\BibSpec{collection.article}{%
    +{}  {\PrintAuthors}                {author}
    +{,} { \textit}                     {title}
    +{.} { }                            {part}
    +{:} { \textit}                     {subtitle}
    +{,} { \PrintContributions}         {contribution}
    +{,} { \PrintConference}            {conference}
    +{}  {\PrintBook}                   {book}
    +{,} { }                            {booktitle}
    +{,} { \PrintDateB}                 {date}
    +{,} { pp.~}                        {pages}
    +{,} { }                            {publisher}
    +{,} { }                            {organization}
    +{,} { }                            {address}
    +{,} { }                            {status}
    +{,} { \PrintDOI}                   {doi}
    +{,} { \eprint}        {eprint}
    +{}  { \parenthesize}               {language}
    +{}  { \PrintTranslation}           {translation}
    +{;} { \PrintReprint}               {reprint}
    +{.} { }                            {note}
    +{.} {}                             {transition}
    +{}  {\SentenceSpace \PrintReviews} {review}
}
\BibSpec{misc}{%
  +{}{\PrintAuthors}  {author}
  +{,}{ \textit}      {title}
  +{.}{ }             {how}
  +{}{ \parenthesize} {date}
  +{,} { available at \eprint}        {eprint}
  +{,}{ available at \url}{url}
  +{,}{ }             {note}
  +{.}{}              {transition}
}

\numberwithin{equation}{section}

\newtheorem{theorem}{Theorem}[section]
\newtheorem{corollary}[theorem]{Corollary}
\newtheorem{lemma}[theorem]{Lemma}
\newtheorem{proposition}[theorem]{Proposition}
\theoremstyle{remark}
\newtheorem{remark}[theorem]{Remark}

\theoremstyle{definition}
\newtheorem{definition}[theorem]{Definition}

\newcommand{\bp}{\begin{proof}}
\newcommand{\ep}{\end{proof}}
\DeclareMathOperator{\Nat}{Nat}
\DeclareMathOperator{\Irr}{Irr}
\DeclareMathOperator{\Mat}{Mat}
\DeclareMathOperator{\supp}{supp}
\DeclareMathOperator{\Rep}{Rep}
\DeclareMathOperator{\Ind}{Ind}
\DeclareMathOperator{\Endb}{End_b}
\DeclareMathOperator{\Endd}{End}
\DeclareMathOperator{\tr}{tr}
\DeclareMathOperator{\Tr}{Tr}

\newcommand{\C}{\mathbb{C}}
\newcommand{\N}{\mathbb{N}}
\newcommand{\R}{\mathbb{R}}

\newcommand{\Z}{\mathbb{Z}}
\newcommand{\A}{\mathcal{A}}
\newcommand{\BB}{\mathcal{B}}
\newcommand{\CC}{\mathcal{C}}
\newcommand{\E}{\mathcal{E}}

\newcommand{\HH}{\mathcal{H}}
\newcommand{\M}{\mathcal{M}}
\newcommand{\vNN}{\mathcal{N}}
\newcommand{\PP}{\mathcal{P}}
\newcommand{\RR}{\mathcal{R}}
\newcommand{\DD}{\mathcal{D}}

\newcommand{\Hilb}{\mathrm{Hilb}}
\newcommand{\st}{\mathrm{st}}
\newcommand{\vNotimes}{\mathbin{\bar{\otimes}}}

\newcommand{\Tens}{\mathpzc{Tens}}
\newcommand{\un}{\mathds{1}}
\newcommand{\norm}[1]{\left\| #1 \right \|}
\newcommand{\dmin}{d_{\mathrm{min}}}
\newcommand{\Dhat}{{\hat\Delta}}
\mathchardef\mhyph="2D
\newcommand{\NN}{{N\mhyph N}}
\newcommand{\End}[2]{#1\left(#2\right)}

\begin{document}

\title{Poisson boundaries of monoidal categories}

\author[S. Neshveyev]{Sergey Neshveyev}

\email{sergeyn@math.uio.no}

\address{Department of Mathematics, University of Oslo, P.O. Box 1053
Blindern, NO-0316 Oslo, Norway}

\thanks{The research leading to these results has received funding
from the European Research Council under the European Union's Seventh
Framework Programme (FP/2007-2013)
/ ERC Grant Agreement no. 307663%--NCGQG
}

\author[M. Yamashita]{Makoto Yamashita}

\email{yamashita.makoto@ocha.ac.jp}

\address{Department of Mathematics, Ochanomizu University, Otsuka
2-1-1, Bunkyo, 112-8610 Tokyo, Japan}

\thanks{Supported by JSPS KAKENHI Grant Number 25800058}

\date{May 26, 2014; minor corrections July 13, 2016}



\maketitle

\begin{minipage}{\linewidth}
\begin{abstract}
Given a rigid C$^*$-tensor category $\CC$ with simple unit and a probability measure $\mu$ on the set of isomorphism classes of its simple objects, we define the Poisson boundary of $(\CC,\mu)$. This is a new C$^*$-tensor category $\PP$, generally with nonsimple unit, together with a unitary tensor functor $\Pi\colon\CC\to\PP$. Our main result is that if $\PP$ has simple unit (which is a condition on some classical random walk), then $\Pi$ is a universal unitary tensor functor defining the amenable dimension function on $\CC$. Corollaries of this theorem unify various results in the literature on amenability of C$^*$-tensor categories, quantum groups, and subfactors.
\end{abstract}
\selectlanguage{french}
\begin{abstract}
Etant donn\'{e}es une C$^*$-cat\'{e}gorie tensorielle rigide $\CC$ dont l'objet unit\'{e} est simple ainsi qu'une mesure de probabilit\'{e} $\mu$ sur l'ensemble de classes d'isomorphisme des objets simples, nous d\'{e}finissons la fronti\`{e}re de Poisson de $(\CC, \mu)$. C'est une nouvelle C$^*$-cat\'{e}gorie tensorielle $\PP$ dont l'objet unit\'{e} n'est pas, en g\'{e}n\'{e}ral, simple, coupl\'{e}e avec un foncteur unitaire tensoriel $\Pi\colon\CC\to\PP$.  Notre r\'{e}sultat principal assure que si l'objet unit\'{e} de $\PP$ est simple (ce qui se traduit par une condition sur une certaine marche al\'{e}atoire classique), alors $\Pi$ est un foncteur unitaire tensoriel universel qui d\'{e}finit la fonction de dimension moyennable sur $\CC$. Les corollaires de ce th\'{e}or\`{e}me unifient diff\'{e}rents r\'{e}sultats connus sur la moyennabilit\'{e} des C$^*$-cat\'{e}gories tensorielles, des groupes quantiques et des sous-facteurs.
\end{abstract}
\selectlanguage{english}
\end{minipage}

\bigskip

\section*{Introduction}

The notion of amenability for monoidal categories first appeared in
Popa's seminal work~\cite{MR1278111} on classification of subfactors
as a crucial condition defining a class of inclusions admitting good
classification. He then gave various characterizations of this property
analogous to the usual amenability conditions for discrete groups: a
Kesten type condition on the norm of the principal graph, a F{\o}lner
type condition on the existence of almost invariant sets, and a
Shannon--McMillan--Breiman type condition on relative entropy, to name
a few.

This stimulated a number of interesting developments in related fields
of operator algebras. First, Longo and Roberts~\cite{MR1444286}
developed a general theory of dimension for C$^*$-tensor categories,
and indicated that the language of sectors/subfactors is well suited
for studying amenability in this context.  Then Hiai and
Izumi~\cite{MR1644299} studied amenability for fusion
algebras/hypergroups endowed with a probability measure, and obtained
many characterizations of this property in terms of random walks and
almost invariant vectors in the associated $\ell^p$-spaces. These
studies were followed by the work of Hayashi and
Yamagami~\cite{MR1749868}, who established a way to realize amenable
monoidal categories as bimodule categories over the hyperfinite II$_1$
factor.

In addition to subfactor theory, another source of interesting
monoidal categories is the theory of quantum groups. In this
framework, the amenability question concerns the existence of almost
invariant vectors and invariant means for a discrete quantum group, or
some property of the dimension function on the category of unitary
representations of a compact quantum group~\citelist{\cite{MR1679171}\cite{MR2276175} \cite{MR2113848}}.  Here, one should be aware that there are two different notions of amenability involved. One is coamenability of compact quantum groups (equivalently, amenability of their discrete duals) considered in the regular representations, the
other is amenability of representation categories. These notions
coincide only for quantum groups of Kac type.

\smallskip

In yet another direction, Izumi~\cite{MR1916370} developed a theory of
noncommutative Poisson boundaries for discrete quantum groups in order to study
the minimality (or lack thereof) of infinite tensor product type
actions of compact quantum groups. From the subsequent
work~\citelist{\cite{MR2200270}\cite{MR2335776}} it became
increasingly clear that for coamenable compact quantum groups the
Poisson boundary captures a very elaborate difference between the two
amenability conditions. Later, an important result on noncommutative
Poisson boundaries was obtained by De Rijdt and Vander
Vennet~\cite{MR2664313}, who found a way to compute the boundaries
through monoidal equivalences.  In light of the categorical duality
for compact quantum group actions recently developed
in~\citelist{\cite{MR3121622}\cite{MR3426224}},
this result suggests that the Poisson boundary should really be an
intrinsic notion of the representation category $\Rep G$ itself,
rather than of the choice of a fiber functor giving a concrete
realization of $\Rep G$ as a category of Hilbert spaces. Starting from
this observation, in this paper we define Poisson boundaries for
monoidal categories.

To be more precise, our construction takes a rigid C$^*$-tensor
category $\CC$ with simple unit and a probability measure $\mu$ on the
set $\Irr(\CC)$ of isomorphism classes of simple objects, and gives
another C$^*$-tensor category $\PP$ together with a unitary tensor
functor $\Pi\colon\CC\to\PP$. Although the category~$\PP$ is defined
purely categorically, there are several equivalent ways to describe
it, or at least its morphism sets, that are more familiar to the operator
algebraists. One is an analogue of the standard description of
classical Poisson boundaries as ergodic components of the time
shift. Another is in terms of relative commutants of von Neumann
algebras, in the spirit
of~\citelist{\cite{MR1444286}\cite{MR1749868}\cite{MR1916370}}. For
categories arising from subfactors and quantum groups, this can be
made even more concrete. For subfactors, computing the Poisson boundary essentially corresponds to passing to the standard model of a subfactor~\cite{MR1278111}. For quantum groups, not surprisingly as this was our initial motivation, the Poisson boundary of the representation category of $G$ can be described in terms of the Poisson boundary of~$\hat G$. The last
result will be discussed in detail in a separate
publication~\cite{MR3291643}, since we also want to describe the
action of $\hat G$ on the boundary in categorical terms and this would
lead us away from the main subject of this paper.

\smallskip

Our main result is that if $\PP$ has simple unit, which corresponds to
ergodicity of the classical random walk defined by $\mu$ on
$\Irr(\CC)$, then $\Pi\colon\CC\to\PP$ is a universal unitary tensor
functor which induces the amenable dimension function on $\CC$. From
this we conclude that $\CC$ is amenable if and only if there exists a
measure $\mu$ such that $\Pi$ is a monoidal equivalence. The last
result is a direct generalization of the famous characterization of
amenability of discrete groups in terms of their Poisson boundaries
due to Furstenberg~\cite{MR0352328}, Kaimanovich and
Vershik~\cite{MR704539}, and Rosenblatt~\cite{MR630645}. From this
comparison it should be clear that, contrary to the usual
considerations in subfactor theory, it is not enough to work only with
finitely supported measures, since there are amenable groups which do
not admit any finitely supported ergodic
measures~\cite{MR704539}. The characterization of amenability in
terms of Poisson boundaries generalizes several results
in~\citelist{\cite{MR1278111} \cite{MR1444286}\cite{MR1749868}}. Our
main result also allows us to describe functors that factor through
$\Pi$ in terms of categorical invariant means. For quantum groups
this essentially reduces to the equivalence between coamenability of
$G$ and amenability of~$\hat
G$~\citelist{\cite{MR2276175}\cite{MR2113848}}.

\smallskip

Although our theory gives a satisfactory unification of various
amenability results, the main remarkable property of the functor
$\Pi\colon\CC\to\PP$ is, in our opinion, the universality. If the
category $\PP$ happens to have a simpler structure compared to $\CC$,
this universality allows one to reduce classification of functors from
$\CC$ inducing the amenable dimension function to an easier
classification problem for functors from $\PP$. This idea will be used
in~\cite{MR3556413} to classify a class of compact quantum
groups.

\smallskip

\paragraph{\bf Acknowledgement} M.Y.~thanks M.~Izumi, S.~Yamagami, T.~Hayashi, and R.~Tomatsu for their interest and encouragement at various stages of
the project.

\bigskip

\section{Preliminaries}
\label{sec:preliminaries}

\subsection{Monoidal categories}
\label{sec:monoidal-categories}

In this paper we study \emph{rigid C$^*$-tensor
categories}.  By now there are many texts covering the basics of
this subject, see for
example~\citelist{\cite{MR2091457}\cite{MR2681261}\cite{MR3204665}}
and references therein.  We mainly follow the conventions
of~\cite{MR3204665}, but for the convenience of the reader
we summarize the basic definitions and facts below.

\smallskip

A \emph{C$^*$-category} is a category $\CC$ whose morphism sets
$\CC(U, V)$ are complex Banach spaces endowed with complex conjugate
involution $\CC(U, V) \to \CC(V, U)$, $T \mapsto T^*$ satisfying the
C$^*$-identity. Unless said otherwise, we always assume that $\CC$ is
closed under finite direct sums and subobjects. The latter means that any
idempotent in the endomorphism ring $\End{\CC}{X} = \CC(X, X)$ comes
from a direct summand of~$X$.

A C$^*$-category is said to be \emph{semisimple} if any object is
isomorphic to a direct sum of simple (that is, with the endomorphism
ring $\C$) objects. We then denote the isomorphism classes of simple
objects by~$\Irr (\CC)$ and assume that this set is at most countable.
Many results admit formulations which do not require this
assumption and can be proved by considering subcategories generated by
countable sets of simple objects, but we leave this matter to the
interested reader.

A \emph{unitary functor}, or a \emph{C$^*$-functor}, is a linear
functor of C$^*$-categories $F\colon \CC \to \CC'$ satisfying $F(T^*)
= F(T)^*$.

In this paper we frequently perform the following operation: starting
from a C$^*$-category $\CC$, we replace the morphisms sets by some
larger system $\DD(X, Y)$ naturally containing the original $\CC(X,
Y)$.  Then we perform the \emph{idempotent completion} to construct a
new category $\DD$. That is, we regard the projections $p
\in \End{\DD}{X}$ as objects in the new category, and take $q \DD(X,
Y) p$ as the morphism set from the object represented by $p
\in \End{\DD}{X}$ to the one by $q \in \End{\DD}{Y}$.  Then the
embeddings $\CC(X, Y) \to \DD(X, Y)$ can be considered as a
C$^*$-functor $\CC \to \DD$.

A \emph{C$^*$-tensor category} is a C$^*$-category endowed with a
unitary bifunctor $\otimes \colon \CC \times \CC \to \CC$, a
distinguished object $\un \in \CC$, and natural unitary isomorphisms
\begin{align*} \un \otimes U &\simeq U \simeq U \otimes \un,& \Phi({U,
V, W})&\colon (U \otimes V) \otimes W \to U \otimes (V \otimes W)
\end{align*} satisfying certain compatibility conditions.

A \emph{unitary tensor functor}, or a \emph{C$^*$-tensor functor},
between two C$^*$-tensor categories $\CC$ and $\CC'$ is given by a
triple $(F_0, F, F_2)$, where $F$ is a C$^*$-functor $\CC \to \CC'$,
$F_0$ is a unitary isomorphism $\un_{\CC'} \to F(\un_\CC)$, and $F_2$
is a natural unitary isomorphism $F(U) \otimes F(V) \to F(U \otimes
V)$, which are compatible with the structure morphisms of $\CC$ and
$\CC'$. As a rule, we denote tensor functors by just one symbol
$F$.

When $\CC$ is a strict C$^*$-tensor category and $U \in \CC$, an
object $V$ is said to be a \emph{dual object} of $U$ if there are
morphisms $R \in \CC(\un, V \otimes U)$ and $\bar{R} \in \CC(\un, U
\otimes V)$ satisfying the conjugate equations
\begin{align*} (\iota_{V} \otimes \bar{R}^*) (R \otimes \iota_{V}) &=
\iota_{V},& (\iota_{U} \otimes R^*) (\bar{R} \otimes \iota_{U}) &=
\iota_{U}.
\end{align*} If any object in $\CC$ admits a dual, $\CC$ is said to be
\emph{rigid} and we denote a choice of a dual of $U \in \CC$
by~$\bar{U}$. We assume that rigid C$^*$-tensor categories have simple tensor units.

A rigid C$^*$-tensor category (with simple unit) has
finite dimensional morphism spaces and hence is automatically
semisimple by our assumption of existence of subobjects.

The quantity
$$
d^\CC(U) =\min_{(R, \bar{R})} \norm{R} \norm{\bar{R}}
$$
is called the \emph{intrinsic dimension} of $U$, where $(R, \bar{R})$
runs through the set of solutions of conjugate equations as above. We
omit the superscript $\CC$ when there is no danger of confusion.  A
solution $(R, \bar{R})$ of the conjugate equations for $U$ is called
\emph{standard} if
$$
\|R\|=\|\bar R\|=d(U)^{1/2}.
$$
Solutions of the conjugate equations for $U$ are unique up to the
transformations $$(R, \bar{R}) \mapsto ((T^* \otimes \iota) R, (\iota
\otimes T^{-1}) \bar{R}).$$ Furthermore, if $(R,\bar R)$ is standard,
then such a transformation defines a standard solution if and only if
$T$ is unitary.

In a rigid C$^*$-tensor category $\CC$ we often fix standard solutions
$(R_U,\bar R_U)$ of the conjugate equations for every object $U$. Then
$\CC$ becomes \emph{spherical} in the sense that one has the equality
$R_U^* (\iota \otimes T) R_U = \bar{R}_U^* (T \otimes \iota)
\bar{R}_U$ for any $T \in \End{\CC}{U}$.  The normalized linear
functional
$$
\tr_U(T) = d(U)^{-1} R_U^* (\iota \otimes T)
R_U=d(U)^{-1}\bar{R}_U^* (T \otimes \iota) \bar{R}_U
$$
is a tracial state on the finite dimensional C$^*$-algebra
$\End{\CC}{U}$. It is independent of the choice of a standard
solution. More generally, for any objects~$X$, $U$ and~$V$ we can consider the normalized partial categorical traces
$$
\tr_X\otimes\iota\colon\CC(X\otimes U,X\otimes V)\to\CC(U,V)
\ \ \text{and}\ \ \iota\otimes\tr_X\colon\CC(U\otimes X,V\otimes
X)\to\CC(U,V).
$$
Namely, with a standard solution
$(R_X, \bar{R}_X)$ as above, we have
$$
(\tr_X\otimes\iota)(T)=d(X)^{-1}(R_X^*\otimes\iota)(\iota\otimes
T)(R_X\otimes\iota),
\ \ (\iota\otimes\tr_X)(T)=d(X)^{-1}(\iota\otimes\bar
R_X^*)(T\otimes\iota)(\iota\otimes\bar R_X).
$$

Given a rigid C$^*$-tensor category $\CC$, if $[U]$ and $[V]$ are
elements of $\Irr(\CC)$, we can define their product in
$\Z_+[\Irr(\CC)]$ by putting
$$
[U] \cdot [V] = \sum_{[W] \in \Irr (\CC)} \dim \CC(W, U \otimes V)
[W],
$$
thus getting a semiring $\Z_+[\Irr(\CC)]$. Extending this formula by
bilinearity, we obtain a ring structure on $\Z[\Irr(\CC)]$. The map
$[U] \mapsto d(U)$ extends to a ring homomorphism $\Z[\Irr(\CC)] \to
\R$.  The pair $(\Z[\Irr(\CC)], d)$ is called the \emph{fusion
algebra} of $\CC$.  In general, a ring homomorphism $d'\colon
\Z[\Irr(\CC)] \to \R$ satisfying $d'([U]) > 0$ and $d'([U])=d'([\bar
U])$ for every $[U] \in \Irr(\CC)$ is said to be a \emph{dimension
function} on $\CC$.

For a rigid C$^*$-tensor category $\CC$, the right multiplication by
$[U] \in \Irr(\CC)$ on $\Z[\Irr(\CC)]$ can be considered as a densely
defined operator $\Gamma_U$ on $\ell^2(\Irr(\CC))$.  This definition
extends to arbitrary objects of $\CC$ by the formula $\Gamma_U =
\sum_{[V] \in \Irr(\CC)} \dim(V, U) \Gamma_V$.  If $d'$ is a dimension
function on $\CC$, one has the estimate $$\norm{\Gamma_U}_{B(\ell^2(
\Irr(\CC)))} \le d'(U).$$ If the equality holds for all objects $U$,
then the dimension function $d'$ is called \emph{amenable}. Clearly,
there can be at most one amenable dimension function. If the intrinsic
dimension function is amenable, then $\CC$ itself is called amenable.

\subsection{Categories of functors} \label{sstensorfunctors}

Given a rigid C$^*$-tensor category $\CC$ we will consider the
category of unitary tensor functors from $\CC$ into C$^*$-tensor
categories. Its objects are pairs $(\A,E)$, where $\A$ is a
C$^*$-tensor category and $E\colon\CC\to\A$ is a unitary tensor
functor. The morphisms $(\A,E)\to(\BB,F)$ are unitary tensor functors
$G\colon\A\to\BB$, considered up to natural unitary monoidal
isomorphisms,\footnote{Therefore the category of functors from $\CC$
we consider here is different from the category $\Tens(\CC)$ defined
in \cite{MR3291643}, where we wanted to distinguish between
isomorphic functors and defined a more refined notion of morphisms.}
such that $GE$ is naturally unitarily isomorphic to $F$.

A more concrete way of thinking of this category is as follows. First
of all we may assume that~$\CC$ is strict. Consider a unitary tensor
functor $E\colon\CC\to\A$. The functor $E$ is automatically faithful
by semisimplicity and existence of conjugates in $\CC$. It follows
that by replacing the pair $(\A,E)$ by an isomorphic one, we may
assume that $\A$ is a strict C$^*$-tensor category containing $\CC$
and $E$ is simply the embedding functor. Namely, define the new sets
of morphisms between objects $U$ and $V$ in $\CC$ as $\A(E(U),E(V))$,
and then complete the category we thus obtain with respect to
subobjects.

Assume now that we have two strict C$^*$-tensor categories $\A$ and
$\BB$ containing $\CC$, and let $E\colon\CC\to \A$ and
$F\colon\CC\to\BB$ be the embedding functors. Assume $[G]\colon
(\A,E)\to(\BB,F)$ is a morphism. This means that there exist unitary
isomorphisms $\eta_U\colon G(U)\to U$ in $\BB$ such that
$G(T)=\eta_V^{-1}T\eta_U$ for any morphism $T\in\CC(U,V)$, and the
morphisms $$G_{2}(U,V)\colon G(U)\otimes G(V)\to G(U\otimes V)$$
defining the tensor structure of $G$ restricted to $\CC$ are given by
$G_{2}(U,V)=\eta_{U\otimes V}^{-1}(\eta_U\otimes\eta_V)$. For
objects~$U$ of~$\A$ that are not in $\CC$ put
$\eta_U=1\in\BB(G(U))$. We can then define a new unitary tensor
functor $\tilde G\colon \A\to\BB$ by letting $\tilde G(U)=U$ for
objects $U$ in $\CC$ and $\tilde G(U)=G(U)$ for the remaining objects,
$\tilde G(T)=\eta_VG(T)\eta_U^{-1}$ for morphisms, and $\tilde
G_2(U,V)=\eta_{U\otimes
V}G_2(U,V)(\eta_U^{-1}\otimes\eta_V^{-1})$.  Then $[G]=[\tilde G]$ and
the restriction of $\tilde G$ to $\CC\subset\A$ coincides with the
embedding (tensor) functor $\CC\to\BB$.

Therefore, any unitary tensor functor $\CC\to\A$ is naturally unitary
isomorphic to an embedding functor, and the morphisms between two such
embeddings $E\colon\CC\to\A$ and $F\colon\CC\to\BB$ are the unitary
tensor functors $G\colon\A\to\BB$ extending $F$, considered up to
natural unitary isomorphisms. If, furthermore, $\A$ is generated by
the objects of $\CC$ then $[G]$ is completely determined by the maps
$\A(U,V)\to\BB(U,V)$ extending the identity maps on $\CC(U,V)$ for all
objects $U$ and $V$ in $\CC$.

\subsection{Subfactor theory}

Let $N \subset M$ be an inclusion of von Neumann algebras represented
on a Hilbert space $H$.  There is a canonical bijective correspondence
between the normal semifinite faithful operator valued weights $\Phi
\colon M \to N$ and the ones $\Psi\colon N' \to M'$ in terms of
spatial derivatives~\cite{MR561983}.  Namely, for every $\Phi$ there
is a unique $\Psi$ denoted by $\Phi^{-1}$ and characterized by
the equation
$$
\frac{d \omega \Phi}{d \omega'} = \frac{d \omega}{d \omega'
\Phi^{-1}},
$$
where $\omega$ and $\omega'$ are any choices of normal semifinite
faithful weights on $N$ and $M'$.

If $E$ is a normal faithful conditional expectation from $M$ to $N$,
its \emph{index} $\Ind E$ can be defined as
$E^{-1}(1)$~\cite{MR829381}.  Suppose that $M$ and $N$ are factors
admitting conditional expectations of finite index.  Then the index is
a positive scalar and there is a unique choice of $E$ which minimizes
$\Ind E$.  This $E$ is called the \emph{minimal conditional
  expectation} of the subfactor $N \subset M$~\cite{MR976765}.

Suppose that $N \subset M$ is a subfactor endowed with a normal
conditional expectation of finite index $E\colon M \to N$.  We then
obtain a von Neumann algebra $M_1$ called the \emph{basic extension}
of $N\subset M$ with respect to $E$, as follows.  Taking a normal
semifinite faithful weight $\psi$ on $N$, the algebra $M_1 \subset
B(L^2(M, \psi E))$ is generated by~$M$ and the orthogonal projection
$e_N$, called the Jones projection, onto $L^2(N, \psi) \subset L^2(M,
\psi E)$.  One has the equality $M_1 = J N' J$, where $J$ is the
modular conjugation of $M$ with respect to~$\psi E$.  From the above
correspondence of operator valued weights, there is a canonical
conditional expectation $E_1 \colon M_1 \to M$ which has the same
index as $E$, namely, $E_1=(\Ind E)^{-1}JE^{-1}(J\cdot
J)J$.  Iterating this procedure, we obtain a tower of von Neumann
algebras
$$
N \subset M \subset M_1 \subset M_2 \subset \cdots.
$$
The higher relative commutants
$$
N' \cap M_k = \{ x \in M_k \mid \forall y \in N \colon x y = y x \}
$$
are finite dimensional C$^*$-algebras, with bound $\dim(N' \cap M_k)
\le (\Ind E)^k$.  The algebras $M' \cap M_{2 k}$
($k \in \N$) can be considered as the endomorphism rings of $M
\otimes_N M \otimes_N \cdots \otimes_N M$ in the category of
$M$-bimodules, and there are similar interpretations for the algebras $N' \cap
M_{2k + 1}$, etc., in terms of $N$-bimodules, $M$-$N$-modules, and
$N$-$M$-modules.

\subsection{Relative entropy}

An important numerical invariant for inclusions of von Neumann
algebras, closely related to index, is relative entropy. For this part
we follow the exposition in~\cite{MR2251116}.

When $\varphi$ and
$\psi$ are positive linear functionals on a C$^*$-algebra $M$, we
denote their relative entropy by $S(\varphi,\psi)$. If $M$ is finite
dimensional, it can be defined as
$$
S(\varphi,\psi)=\begin{cases}\Tr(Q_\varphi(\log Q_\varphi-\log
Q_\psi)),& \text{if}\ \ \varphi\le\lambda\psi\ \ \text{for some}\ \
\lambda>0,\\+\infty,&\text{otherwise,}\end{cases}
$$
where $\Tr$ is the canonical trace on $M$ which takes value $1$ on
every minimal projection in $M$, and $Q_\varphi\in M$ is the density
matrix of $\varphi$, so that we have $\varphi(x) =\Tr( x
Q_\varphi)$.  For a single positive linear functional $\psi$ on a
finite dimensional $M$, we also have its von Neumann entropy defined
as $S(\psi)=-\Tr(Q_\psi\log Q_\psi)$.

\smallskip

Given an inclusion of C$^*$-algebras $N\subset M$ and a state
$\varphi$ on $M$, the \emph{relative entropy} $H_\varphi(M|N)$ (also
called \emph{conditional entropy} in the classical probability
theory) is defined as the supremum of the quantities
$$
\sum_i(S(\varphi_i,\varphi)-S(\varphi_i|_N,\varphi|_N))
$$
where $(\varphi_i)_i = (\varphi_1, \ldots, \varphi_k)$ runs through
the tuples of positive linear functionals on $M$ satisfying $\varphi =
\sum^k_{i=1} \varphi_i$. If $M$ is finite dimensional, this can also
be written as
$$
H_\varphi(M | N) = S(\varphi) - S(\varphi|_N) + \sup_{(\varphi_i)_i}
\sum_i \bigl(S(\varphi_i|_N) - S(\varphi_i)\bigr),
$$
where supremum is again taken over all finite decompositions of
$\varphi$.

Relative entropy has the following lower semicontinuity
property.  Suppose that $N \subset M$ is an inclusion of von Neumann
algebras and $\varphi$ is a normal state on $M$.  Suppose that $B_i
\subset A_i$ ($i=1,2,\ldots$) are increasing sequences of subalgebras
$B_i \subset N$, $A_i \subset M$ such that $\bigcup_i A_i$ and $\bigcup_i
B_i$ are $s^*$-dense in~$M$ and~$N$, respectively.  Then one has
the estimate $$H_\varphi(M | N) \le \liminf_i H_\varphi(A_i | B_i).$$

If $N \subset M$ is an inclusion of von Neumann algebras and $E \colon
M \to N$ is a normal conditional expectation, the relative entropy of
$M$ and $N$ with respect to $E$ is defined by
$$
H_E(M | N) = \sup_\varphi H_\varphi(M | N),
$$
where $\varphi$ runs through the normal states on $M$ satisfying
$\varphi = \varphi E$~\cite{MR1096438}. If $M$ and $N$ are factors,
then we have the estimate $H_E(M|N)\le\log\Ind E$.

\bigskip

\section{Categorical Poisson boundary}

Let $\CC$ be a strict rigid C$^*$-tensor category satisfying our
standard assumptions: it is closed under finite direct sums and subobjects,
the tensor unit is simple, and $\Irr(\CC)$ is at most countable.

\smallskip

Let $\mu$ be a probability measure on $\Irr(\CC)$. The Poisson
boundary of $(\CC,\mu)$ will be a new C$^*$-tensor category $\PP$,
possibly with nonsimple unit, together with a unitary tensor functor
$\Pi\colon\CC\to\PP$. In this section we define $(\PP,\Pi)$ in purely
categorical terms. In the next section we will give several more
concrete descriptions of this construction.

\smallskip

For an object $U$ consider the functor $\iota\otimes
U\colon\CC\to\CC$, $X \mapsto X \otimes U$. Given two objects $U$ and $V$, consider the space
$\Nat(\iota\otimes U,\iota\otimes V)$ of natural transformations from
$\iota\otimes U$ to $\iota\otimes V$, so elements of
$\Nat(\iota\otimes U,\iota\otimes V)$ are collections
$\eta=(\eta_X)_X$ of morphisms $\eta_X\colon X\otimes
U\to X\otimes V$, natural in $X$. For every object $X$ we can define a linear operator
$P_X$ on $\Nat(\iota\otimes U,\iota\otimes V)$ by
$$
P_X(\eta)_Y=(\tr_X\otimes\iota)(\eta_{X\otimes Y})
$$
with the partial categorical trace introduced in
Section~\ref{sec:monoidal-categories}.  Denote by
$\hat{\CC}(U, V) \subset \Nat(\iota\otimes U,\iota\otimes V)$ the
subspace of bounded natural transformations, that is, of elements
$\eta$ such that $\sup_Y\|\eta_Y\|<\infty$. More concretely, taking a
representative $U_s$ for each $s \in \Irr(\CC)$, we can present
$\hat{\CC}(U, V)$ as
$$
\hat{\CC}(U, V) \cong \ell^\infty\text{-}\bigoplus_{s}\CC(U_s \otimes U, U_s\otimes V),
$$
since the natural transformations are determined by their actions on the simple objects. This is a Banach space, and the operator
$P_X$ defines a contraction on it. It is also clear that the operator
$P_X$ depends only on the isomorphism class of $X$.

From now on let us fix a representative $U_s$ for every
$s\in \Irr(\CC)$ as above. We write $\tr_s$ instead of $\tr_{U_s}$,
$P_s$ instead of $P_{U_s}$, and so on. Similarly, for a natural
transformation $\eta\colon\iota\otimes U\to\iota\otimes V$ we
write $\eta_s$ instead of $\eta_{U_s}$. Let also denote by
$e\in\Irr(\CC)$ the index corresponding to $\un$. For convenience we
assume that $U_e=\un$. Define an involution on $\Irr(\CC)$ such that
$U_{\bar s}$ is a dual object to $U_s$.

Consider now the operator
$$
P_\mu=\sum_s\mu(s)P_s.
$$
This is a well-defined contraction on $\hat{\CC}(U, V)$. We say that a
bounded natural transformation $\eta\colon\iota\otimes
U\to\iota\otimes V$ is $P_\mu$-{\em harmonic} if
$$P_\mu(\eta)=\eta.$$
Any morphism $T\colon U\to V$ defines a bounded natural transformation
$(\iota_X\otimes T)_X$, which is obviously $P_\mu$-harmonic for every
$\mu$. When there is no ambiguity, we denote this natural
transformation simply by $T$.

The composition of harmonic transformations is in general not
harmonic. But we can define a new composition as follows.

\begin{proposition} \label{pproduct} Given bounded $P_\mu$-harmonic
  natural transformations $\eta\colon\iota\otimes U\to\iota\otimes V$
  and $\nu\colon\iota\otimes V\to\iota\otimes W$, the limit
$$
(\nu\cdot\eta)_X=\lim_{n\to\infty}P^n_\mu(\nu\eta)_X
$$
exists for all objects $X$ and defines a bounded $P_\mu$-harmonic
natural transformation $\iota\otimes U\to\iota\otimes W$. Furthermore,
the composition $\cdot$ is associative.
\end{proposition}

Note that since the spaces $\CC(X\otimes U,X\otimes W)$ are finite
dimensional by our assumptions on $\CC$, the notion of a limit is
unambiguous.

\bp[Proof of Proposition~\ref{pproduct}] This is an immediate
consequence of results of Izumi~\cite{MR2995726} (another proof will
be given in Section~\ref{sec:poiss-bound-from}). Namely, replacing
$U$, $V$ and $W$ by their direct sum we may assume that
$U=V=W$. Then
$$
\hat{\CC}(U) = \hat{\CC}(U, U) \cong
\ell^\infty\text{-}\bigoplus_s\End{\CC}{U_s\otimes U}
$$
is a von Neumann algebra and $P_\mu$ is a normal unital completely
positive map on it. By \cite{MR2995726}*{Corollary~5.2} the subspace
of $P_\mu$-invariant elements is itself a von Neumann algebra with
product $\cdot$ such that $x\cdot y$ is the $s^*$-limit of the
sequence $\{P^n_\mu(xy)\}_n$.  \ep

Using this product on harmonic elements we can define a new
C$^*$-tensor category $\PP=\PP_{\CC,\mu}$ and a unitary tensor functor
$\Pi=\Pi_{\CC,\mu}\colon\CC\to\PP$ as follows.

First consider the category $\tilde\PP$ with the same objects as in
$\CC$, but define the new spaces $\tilde\PP(U,V)$ of morphisms as the
spaces of bounded $P_\mu$-harmonic natural transformations
$\iota\otimes U\to\iota\otimes V$. Define the composition of morphisms
as in Proposition~\ref{pproduct}. We thus get a C$^*$-category,
possibly without subobjects. Furthermore, the C$^*$-algebras
$\End{\tilde\PP}{U}$ are von Neumann algebras.

Next, we define the tensor product of objects in the same way as in
$\CC$, and define the tensor product of morphisms by
$$
\nu\otimes\eta=(\nu\otimes\iota)\cdot(\iota\otimes\eta).
$$
Here, given $\nu\colon \iota\otimes U\to\iota\otimes V$ and
$\eta\colon\iota\otimes W\to\iota\otimes Z$, the natural
transformation $\nu\otimes\iota_Z\colon\iota\otimes U\otimes
Z\to\iota\otimes V\otimes Z$ is defined by
$$
(\nu\otimes\iota_Z)_X=\nu_X\otimes\iota_Z,
$$
while the natural transformation $\iota_U\otimes\eta\colon\iota\otimes
U\otimes W\to\iota\otimes U\otimes Z$ is defined by
$$
(\iota_U\otimes\eta)_X=\eta_{X\otimes U}.
$$
We remark that $\nu\otimes\iota$ and $\iota\otimes\eta$ are still
$P_\mu$-harmonic due to the identities
$$
P_X(\nu\otimes\iota) = P_X(\nu)\otimes\iota,\ \ P_X(\iota\otimes\eta)
= \iota\otimes P_X(\eta).
$$
Note also that by naturality of $\eta$ we have
$(\nu_X\otimes\iota_Z)\eta_{X\otimes U}=\eta_{X\otimes
  V}(\nu_X\otimes\iota_Z)$, which implies that
$$
\nu\otimes\eta=(\iota\otimes\eta)\cdot(\nu\otimes\iota).
$$
This shows that $\otimes\colon\tilde\PP\times\tilde\PP\to\tilde\PP$ is
indeed a bifunctor. Since $\CC$ is strict, this bifunctor is strictly associative.

Finally, complete the category $\tilde\PP$ with respect to
subobjects. This is our C$^*$-tensor category~$\PP$, possibly with
nonsimple unit. Since $\CC$ is rigid, the category $\PP$ is rigid as well. The unitary tensor functor $\Pi\colon\CC\to\PP$ is
defined in the obvious way: it is the strict tensor functor which is
the identity map on objects and $\Pi(T)=(\iota_X\otimes T)_X$ on
morphisms. We will often omit $\Pi$ and simply consider $\CC$ as a
C$^*$-tensor subcategory of $\PP$.

\begin{definition} The pair $(\PP,\Pi)$ is called the {\em Poisson
    boundary} of $(\CC,\mu)$.  We say that the Poisson boundary is
  trivial if $\Pi\colon\CC\to\PP$ is an equivalence of categories, or
  in other words, for all objects~$U$ and~$V$ in $\CC$ the only
  bounded $P_\mu$-harmonic natural transformations $\iota\otimes
  U\to\iota\otimes V$ are the transformations of the form
  $\eta=(\iota_X\otimes T)_X$ for $T\in\CC(U,V)$.
\end{definition}

The algebra $\PP(\un)$ is determined by the random walk on $\Irr(\CC)$
with transition probabilities
$$
p_\mu(s,t)=\sum_r\mu(r)m^t_{rs}\frac{d(t)}{d(r) d(s)},
$$
where $d(s)=d(U_s)$ and $m^t_{rs}=\dim\CC(U_t,U_r\otimes
U_s)$. Namely, if we identify $\hat\CC(\un)$ with
$$
\ell^\infty\text{-}\bigoplus_s\CC(U_s)=\ell^\infty(\Irr(\CC)),
$$
then the operator $P_\mu$ on $\hat\CC(\un)$ is the Markov operator
defined by $p_\mu$, so $(P_\mu f)(s)=\sum_tp_\mu(s,t)f(t)$. Therefore
$\PP(\un)$ is the algebra of bounded measurable functions on the
Poisson boundary, in the usual probabilistic sense, of the random walk
on $\Irr(\CC)$ with transition probabilities~$p_\mu(s,t)$. We say
that~$\mu$ is \emph{ergodic}, if this boundary is trivial, that is,
the tensor unit of~$\PP$ is simple.

We say that $\mu$ is \emph{symmetric} if $\mu(s)=\mu(\bar s)$ for all
$s$, and that $\mu$ is \emph{generating} if every simple object
appears in the decomposition of $U_{s_1}\otimes\dots \otimes U_{s_n}$
for some $s_1,\dots,s_n\in\supp\mu$ and $n\ge1$. Equivalently, $\mu$
is generating if $\bigcup_{n\ge1}\supp\mu^{*n}=\Irr(\CC)$, where the
convolution of probability measures on $\Irr(\CC)$ is defined by
$$
(\nu*\mu)(t)=\sum_{s,r}\nu(s)\mu(r)m^t_{sr}\frac{d(t)}{d(s) d(r)}.
$$
We will write $\mu^n$ instead of $\mu^{*n}$. The definition of the
convolution is motivated by the identity $P_\mu P_\nu=P_{\nu*\mu}$.

We remark that a symmetric ergodic measure $\mu$, or even an ergodic
measure with symmetric support, is automatically generating. Indeed,
the symmetry assumption implies that we have a well-defined
equivalence relation on $\Irr(\CC)$ such that $s\sim t$ if and only if
$t$ can be reached from $s$ with nonzero probability in a finite
nonzero number of steps. Then any bounded function on $\Irr(\CC)$ that
is constant on equivalence classes is $P_\mu$-harmonic. Hence $\mu$ is
generating by the ergodicity assumption.

\smallskip

Let us say that $\CC$ is \emph{weakly amenable} if the fusion algebra
$(\Z[\Irr(\CC)],d)$ is weakly amenable in the sense of Hiai and
Izumi~\cite{MR1644299}, that is, there exists a left invariant mean on
$\ell^\infty(\Irr(\CC))$. By definition this is a state $m$ such that
$m(P_s(f))=m(f)$ for all $f\in\ell^\infty(\Irr(\CC))$ and
$s\in\Irr(\CC)$. Of course, it is also possible to define right
invariant means, and by \cite{MR1644299}*{Proposition~4.2} if there
exists a left or right invariant mean, then there exists a
bi-invariant mean. By the same proposition amenability implies weak
amenability, as the term suggests. But as opposed to the group case,
in general, the converse is not true. Using this terminology let us
record the following known result.

\begin{proposition} \label{pweakamen} An ergodic probability measure
on $\Irr(\CC)$ exists if and only if $\CC$ is weakly
amenable. Furthermore, if an ergodic measure exists, then it can be
chosen to be symmetric and with support equal to the entire space
$\Irr(\CC)$.
\end{proposition}

\bp If $\mu$ is an ergodic measure, then any weak$^*$ limit point of
the sequence $n^{-1}\sum^{n-1}_{k=0}\mu^k$ defines a right invariant
mean. For random walks on groups this implication was observed by
Furstenberg. The other direction is proved
in~\cite{MR1749868}*{Theorem~2.5}. It is an analogue of a result of
Kaimanovich--Vershik and Rosenblatt.  \ep

It should be remarked that if the fusion algebra of $\CC$ is weakly
amenable and finitely generated, in general it is not possible to find
a finitely supported ergodic
measure~\cite{MR704539}*{Proposition~6.1}.

\smallskip

To finish the section, let us show that, not surprisingly, categorical
Poisson boundaries are of interest only for infinite categories.

\begin{proposition}\label{pfb} Assume $\CC$ is finite, meaning that
$\Irr(\CC)$ is finite, and $\mu$ is generating. Then the Poisson
boundary of $(\CC,\mu)$ is trivial.
\end{proposition}

\bp The proof is similar to the proof of triviality of the Poisson
boundary of a random walk on a finite set based on the maximum
principle. Fix an object $U$ in $\CC$ and assume that
$\eta\in\hat{\CC}(U)$ is positive and $P_\mu$-harmonic. We claim that
if $\eta\ne0$ then there exists a positive nonzero morphism
$T\in\End{\CC}{U}$ such that $\eta\ge T$. Assuming that the claim is
true, we can then choose a maximal $T$ with this property. Applying
again the claim to the element $\eta-T$, we conclude that $\eta=T$ by
maximality.

In order to prove the claim observe that $\eta_e\in\End{\CC}{U}$ is
nonzero. Indeed, by assumption there exists~$s$ such that
$\eta_s\ne0$. Since the categorical traces are faithful, and therefore
partial categorical traces are faithful completely positive maps, it
follows that $P_s(\eta)_e\ne0$. Since $s\in\supp\mu^n$ for some
$n\ge1$, we conclude that $\eta_e=P_{\mu^n}(\eta)_e\ne0$.

Denote the positive nonzero element $\eta_e\in\End{\CC}{U}$ by
$S$. Fix $s\in\Irr(\CC)$. Let $(R_s,\bar R_s)$ be a standard solution
of the conjugate equations for $U_s$, and $p\in \End{\CC}{\bar
U_s\otimes U_s}$ be the projection defined by $p=d(s)^{-1}R_sR_s^*$. By
naturality of $\eta$ we then have $\eta_{\bar U_s\otimes U_s}\ge
p\otimes S$, whence
$$
P_{\bar s}(\eta)_s\ge(\tr_{\bar s}\otimes\iota)(p)\otimes
S=d(s)^{-2}(\iota\otimes S).
$$
Using the generating property of $\mu$ and finiteness of $\Irr(\CC)$,
we conclude that there exists a number $\lambda>0$ such that
$\eta_s\ge \iota\otimes\lambda S$ for all $s$. This proves the
claim.  \ep

\bigskip

\section{Realizations of the Poisson boundary}

As in the previous section, we fix a strict rigid C$^*$-tensor
category $\CC$ and a probability measure~$\mu$ on~$\Irr(\CC)$. In
Sections~\ref{sec:longo-roberts-appr} and~\ref{sHY} we will in
addition assume that $\mu$ is generating. Let $\Pi\colon\CC\to\PP$ be
the Poisson boundary of $(\CC,\mu)$. Our goal is to give several
descriptions of the algebras~$\PP(U)$ of harmonic elements.

\subsection{Time shift on the categorical path space}
\label{sec:poiss-bound-from}

Fix an object $U$. Denote by $M^{(0)}_U$ the von Neumann algebra
$\hat{\CC}(U)\cong\ell^\infty{\text-}\bigoplus_s\End{\CC}{U_s\otimes
U}$. More generally, for every $n\ge0$ consider the von Neumann
algebra
$$
M^{(n)}_U=\Endb(\iota_{\CC^{n+1}}\otimes U),
$$
so $M^{(n)}_U$
consists of bounded collections
$\eta=(\eta_{X_{n},\dots,X_0})_{X_n,\dots,X_0}$ of natural in
$X_n,\dots,X_0$ endomorphisms of $X_n\otimes\dots \otimes X_0\otimes
U$. We consider $M^{(n)}_U$ as a subalgebra of~$M^{(n+1)}_U$ using the
embedding
$$
(\eta_{X_{n},\dots,X_0})_{X_n,\dots,X_0}\mapsto(\iota_{X_{n+1}}\otimes
\eta_{X_n,X_{n-1},\dots,X_0})_{X_{n+1},\dots,X_0}.
$$
Define a conditional
expectation $E_{n+1,n}\colon M^{(n+1)}_U\to M^{(n)}_U$ by
$$
E_{n+1,n}(\eta)_{X_n,\dots,X_0} = \sum_s \mu(s)(\tr_s\otimes\iota)
(\eta_{U_s,X_n,\dots,X_0}).
$$
Taking compositions of such conditional expectations we get normal
conditional expectations $$E_{n,0}\colon M^{(n)}_U\to M^{(0)}_U.$$
These conditional expectations are not faithful for $n\ge1$ unless the
support of $\mu$ is the entire space $\Irr(\CC)$. The support of
$E_{n,0}$ is a central projection, and we denote by $\M^{(n)}_U$ the
reduction of $M^{(n)}_U$ by this projection. More concretely, we have
a canonical isomorphism
\begin{equation}\label{eq:M-n-U-defn}
  \M^{(n)}_U \cong \ell^\infty\text{-}\bigoplus_{\substack{s_n,\dots,s_1\in\supp\mu\\ s_0\in\Irr(\CC)}}
\End{\CC}{U_{s_n}\otimes\dots\otimes U_{s_0}\otimes U}.
\end{equation} The conditional expectations $E_{n,0}$ define normal
faithful conditional expectations
$\E_{n,0}\colon\M^{(n)}_U\to\M^{(0)}_U=M^{(0)}_U$, and similarly
$E_{n+1,n}$ define conditional expectations $\E_{n+1,n}$. Denote by
$\M_U$ the von Neumann algebra obtained as the inductive
limit of the algebras $\M^{(n)}_U$ with respect to $\E_{n,0}$. In other
words, take any faithful normal state $\phi^{(0)}_U$ on $\M^{(0)}_U$. By
composing it with the conditional expectation $\E_{n,0}$ we get a
state $\phi^{(n)}_U$ on $\M^{(n)}_U$. Together these states define a
state on $\bigcup_n\M^{(n)}_U$. Finally, complete $\bigcup_n\M^{(n)}_U$ to a
von Neumann algebra $\M_U$ in the GNS-representation corresponding to this
state. Denote the corresponding normal state on $\M_U$
by~$\phi_U$.

Note that if we start with a trace on
$\M^{(0)}_U$ which is a convex combination of the traces
$\tr_{U_s\otimes U}$, then the corresponding state $\phi_U$
on $\M_U$ is tracial. Since it is faithful on $\M^{(n)}_U$
for every~$n$, it is faithful on~$\M_U$. This shows that
$\M_U$ is a finite von Neumann algebra. Furthermore, the
$\phi_U$-preserving normal faithful conditional expectation
$\E_n\colon\M_U\to \M^{(n)}_U$ coincides with $\E_{n+1,n}$
on $\M^{(n+1)}_U$. It follows that on the dense algebra
$\bigcup_m\M^{(m)}_U$ the conditional expectation $\E_n$ is the limit, in
the pointwise $s^*$-topology, of $\E_{n+1,n}\E_{n+2,n+1}\dots
\E_{m+1,m}$ as $m\to\infty$. Hence $\E_n$ is independent of the choice
of a faithful normal trace $\phi^{(0)}_U$ as above.

\smallskip

Define a unital endomorphism $\theta_U$ of
$\bigcup_nM^{(n)}_U$ such that $\theta_U(M^{(n)}_U)\subset M^{(n+1)}_U$
by
$$
\theta_U(\eta)_{X_{n+1},\dots,X_0}=\eta_{X_{n+1},\dots,X_2,X_1\otimes
X_0}.
$$
Considering $\M^{(k)}_U$ as a quotient of $M^{(k)}_U$ we get a unital
endomorphism of $\bigcup_n\M^{(n)}_U$.

\begin{lemma} \label{lshift}
The endomorphism $\theta_U$ of $\bigcup_n\M^{(n)}_U$
extends to a normal faithful endomorphism of~$\M_U$, which
we continue to denote by $\theta_U$.
\end{lemma}

\bp Consider the normal semifinite faithful (n.s.f.)~trace
$\psi^{(0)}_U=\sum_sd(s)^2\tr_{U_s\otimes U}$ on
$$\M^{(0)}_U\cong\ell^\infty\text{-}\bigoplus_s\End{\CC}{U_s\otimes U}$$ and put
$\psi_U=\psi^{(0)}_U\E_0$. Then $\psi_U$ is an n.s.f.~trace. In order
to prove the lemma it suffices to show that the restriction of
$\psi_U$ to $\bigcup_n\M^{(n)}_U$ is $\theta_U$-invariant. Indeed, if the
invariance holds, then we can define an isometry $U$ on
$L^2(\M_U,\psi_U)$ by
$U\Lambda_{\psi_U}(x)=\Lambda_{\psi_U}(\theta_U(x))$ for
$x\in\bigcup_n\M^{(n)}_U$ such that $\psi_U(x^*x)<\infty$. Let $H\subset
L^2(\M_U,\psi_U)$ be the image of~$U$ and $\M$ be the von Neumann
algebra generated by the image of $\theta_U$. Then $H$ is
$\M$-invariant. We can choose $0\le e_i\le1$ such that
$\theta_U(e_i)\to1$ strongly and $\psi_U(e_i)<\infty$. Now, if
$x\in\M_+$ is such that $x|_H=0$, then
$\psi_U(\theta_U(e_i)x\theta_U(e_i))=0$, and by lower semicontinuity
we get $\psi_U(x)=0$, so $x=0$. Therefore we can define $\theta_U$ as
the composition of the map $\M_U\to B(H)$, $x\mapsto UxU^*$, with the
inverse of the map $\M\to \M|_H$.

It remains to check the invariance. By definition we have
$\E_{n+2,n+1}\theta_U=\theta_U\E_{n+1,n}$ on $\M^{(n+1)}_U$ for all
$n\ge0$. This implies that $\E_{n+1}\theta_U=\theta_U\E_n$ on
$\bigcup_k\M^{(k)}_U$. It follows that for any $x\in \bigcup_n\M^{(n)}_U$ we
have
$$
\psi_U\theta_U(x)=\psi_U\E_1\theta_U(x)=\psi_U\theta_U\E_0(x)
=\psi_U\E_0\theta_U\E_0(x).
$$
This implies that it suffices to show that $\psi_U\E_0\theta_U=\psi_U$
on~$\M^{(0)}_U$. Since $\tr_{U_s\otimes U}=\tr_s(\iota\otimes\tr_U)$,
it is enough to consider the case $U=\un$. Note also that
$\E_0\theta_U=P_\mu$ on~$\M^{(0)}_U$. Thus we have to check that
$\psi_\un P_\mu=\psi_\un$ on
$\M^{(0)}_\un\cong\ell^\infty(\Irr(\CC))$. This is equivalent to the
easily verifiable identity $\mu*m=m$, where
$m=\sum_sd(s)^2\delta_s$.  \ep

We call the endomorphism $\theta_U$ of $\M_U$ the {\em time
shift}. Now, take $\eta\in\M^{(0)}_U$. Then for every $n\ge0$ we can
define an element $\eta^{(n)}\in M^{(n)}_U$ by
$$
\eta^{(n)}_{X_n,\dots,X_0}=\eta_{X_n\otimes\dots\otimes X_0}.
$$
Consider the image of $\eta^{(n)}$ in $\M^{(n)}_U$ and denote it again
by $\eta^{(n)}$, since this is the only element we are interested
in. Then $\eta$ is $P_\mu$-harmonic if and only if
$\E_{1,0}(\eta^{(1)})=\eta$, and in this case
$\E_{n+1,n}(\eta^{(n+1)})=\eta^{(n)}$ for all $n$. Therefore if $\eta$
is $P_\mu$-harmonic, then the sequence $\{\eta^{(n)}\}_n$ is a
martingale. Denote by $\eta^{(\infty)}\in\M_U$ its $s^*$-limit.

\begin{proposition} \label{prop:lr-emb} The map
$\eta\mapsto\eta^{(\infty)}$ is an isomorphism between the von Neumann
algebra $\PP(U)$ of $P_\mu$-harmonic bounded natural transformations
$\iota\otimes U\to\iota\otimes U$ and the fixed point algebra
$\M^{\theta_U}_U$. The inverse map is given by $x\mapsto \E_0(x)$.
\end{proposition}

\bp By definition we have $\eta^{(n)}=\theta^n_U(\eta)$. It follows
that if $\eta$ is $P_\mu$-harmonic, so that
$\eta^{(n)}\to\eta^{(\infty)}$, then the element $\eta^{(\infty)}$ is
$\theta_U$-invariant. We also clearly have
$\E_0(\eta^{(\infty)})=\eta$.

Conversely, take $x\in\M_U^{\theta_U}$. The proof of
Lemma~\ref{lshift} implies that $\E_{n+1}\theta_U=\theta_U\E_n$. Hence
the martingale $\{x_n=\E_n(x)\}_n$ has the property
$x_{n+1}=\theta_U(x_n)$. As $\E_0\theta_U=P_\mu$ on~$\M^{(0)}_U$, we
conclude that $x_0$ is $P_\mu$-harmonic and $x_0^{(\infty)}=x$.

We have thus proved that the maps in the assertion are inverse to
each other. Since they are unital completely positive, they must be
isomorphisms.  \ep

The bijection between $\PP(U)$ and $\M^{\theta_U}_U$ could be used to
give an alternative proof of Proposition~\ref{pproduct}. Namely, we
could define a product $\cdot$ on harmonic elements by
$\nu\cdot\eta=\E_0(\nu^{(\infty)}\eta^{(\infty)})$. Since
$\nu^{(\infty)}\eta^{(\infty)}$ is the $s^*$-limit of the elements
$\nu^{(n)}\eta^{(n)}=(\nu\eta)^{(n)}$, and
$\E_0((\nu\eta)^{(n)})=P^n_\mu(\nu\eta)$, it follows that
$P^n_\mu(\nu\eta)\to \nu\cdot\eta$ in the $s^*$-topology, which is
equivalent to saying that $P^n_\mu(\nu\eta)_X\to (\nu\cdot\eta)_X$ for
every $X$.

\subsection{Relative commutants: Izumi--Longo--Roberts approach}
\label{sec:longo-roberts-appr}

We will now modify the construction of the algebras $\M_U$ to get
algebras $\vNN_U$ and an identification of $\PP(U)$ with
$\vNN_\un'\cap\vNN_U$. Conceptually, instead of considering all paths
of the random walk defined by $\mu$, we consider only paths starting
at the unit object. The time shift is no longer defined on this space,
but by considering a larger space we can still get a description of
$\PP(U)$ in simple von Neumann algebraic terms. For this to work we
have to assume that $\mu$ is generating, so that we can reach any
simple object from the unit.

This identification of harmonic elements is closely related to Izumi's
description of Poisson boundaries of discrete quantum
groups~\cite{MR1916370}. A similar construction was also used by Longo
and Roberts using sector theory~\cite{MR1444286}. More precisely, they
worked with a somewhat limited form of $\mu$ and what we obtain is a
possibly infinite von Neumann algebra for what corresponds to the
finite gauge-invariant von Neumann subalgebra in their work.

\smallskip

We first put $V = \bigoplus_{s \in \supp \mu} U_s$. In the case $\supp
\mu$ is infinite, this should be understood only as a suggestive
notation which does not make sense inside $\CC$. Given an object $U$,
by $\CC(V^{\otimes n}\otimes U)$ we understand the space
$$
\bigoplus_{s_*, s'_* \in \supp \mu^n} \CC(U_{s_n} \otimes \cdots
\otimes U_{s_1} \otimes U, U_{s'_n} \otimes \cdots\otimes U_{s'_1}
\otimes U)
$$
endowed with the obvious $*$-algebra structure. Similarly to
Section~\ref{sec:poiss-bound-from} we have completely positive maps
$$
\E_{n+1,n}=\sum_s \mu(s) (\tr_s \otimes \iota)\colon \CC(V^{\otimes
(n+1)}\otimes U)\to \CC(V^{\otimes n}\otimes U),
$$
and taking the composition of these maps we get maps
$$
\E_{n,0}\colon \CC(V^{\otimes n}\otimes U)\to \CC(U).
$$
Then $\omega_U^{(n)} = \tr_U \E_{n,0}$ is a state on $\CC(V^{\otimes
n}\otimes U)$. We denote by $\vNN^{(n)}_U$ the von Neumann algebra
generated by $\CC(V^{\otimes n}\otimes U)$ in the GNS-representation
defined by this state. The elements of $\vNN_U^{(n)}$ are represented
by certain bounded families in the direct product of the morphism
sets $$\CC(U_{s_n} \otimes \cdots \otimes U_{s_1} \otimes U, U_{s'_n}
\otimes \cdots \otimes U_{s'_1} \otimes U).$$ Since the positive
elements of $\vNN_U^{(n)}$ have positive diagonal entries, the state
$\omega_U^{(n)}$ is faithful on $\vNN_U^{(n)}$.

There is a natural diagonal embedding $\vNN_U^{(n)} \rightarrow
\vNN_U^{(n+1)}$ defined by $T \mapsto \iota_V \otimes T$. The map
$\E_{n+1,n}$ extends then to a normal conditional expectation
$\vNN_U^{(n+1)} \rightarrow \vNN_U^{(n)}$ such that
$\omega^{(n)}_U\E_{n+1,n}=\omega^{(n+1)}_U$. This way we obtain an
inductive system $(\vNN_U^{(n)}, \omega_U^{(n)})_n$ of von Neumann
algebras, and we let $(\vNN_U, \omega_U)$ be the von Neumann algebra
and the faithful state obtained as the limit. As in
Section~\ref{sec:poiss-bound-from}, composing the conditional
expectations $\E_{n+1,n}$ and passing to the limit we get
$\omega_U$-preserving conditional expectations
$\E_n\colon\vNN_U\to\vNN^{(n)}_U$.

\smallskip

When $U = \un$, we simply write $\vNN^{(n)}$ and $\vNN$ instead of
$\vNN^{(n)}_\un$ and $\vNN_\un$. If $U'$ and $U$ are objects in~$\CC$,
then the map $x\mapsto x\otimes\iota_U$ defines an embedding
$\vNN_{U'}\hookrightarrow\vNN_{U'\otimes U}$. In particular, the
algebra $\vNN$ is contained in any of $\vNN_U$.

\smallskip

When $\eta$ is a natural transformation in $\hat\CC(U)$,
the morphism $$\eta_{V^{\otimes n}} = \bigoplus_{s_*} \eta_{U_{s_n}
\otimes \cdots \otimes U_{s_1}}$$ defines an element in the diagonal
part of $\vNN_U^{(n)}$, which we denote by $\eta^{[n]}$.  Note that
the direct summand $s_0 = e$ of~\eqref{eq:M-n-U-defn} can be
identified with the diagonal part of $\vNN_U^{(n)}$, and $\eta^{[n]}$
simply becomes the component of $\eta^{(n)}$ in this summand.  If
$\eta$ is $P_\mu$-harmonic, the sequence $\{\eta^{[n]}\}_n$ forms a
martingale and defines an element $\eta^{[\infty]} \in
\vNN_U^{(\infty)}$.

\begin{proposition}
\label{prop:p-bdry-as-rel-comm-LR0} For every object $U$ in $\CC$, the
map $\eta \mapsto\eta^{[\infty]}$ defines an isomorphism of von
Neumann algebras $\End{\PP}{U}\cong \vNN'\cap\vNN_{U}$.
\end{proposition}

\begin{proof} If $\eta$ is a harmonic element in $\hat\CC(U)$, the
naturality implies that the elements $\eta_{V^{\otimes m}}$ commute
with the image of $\End{\CC}{V^{\otimes n}}$ for $m \ge n$.  Thus,
$\eta^{[\infty]} = \lim_m \eta_{V^{\otimes m}}$ is in the relative
commutant.  Since~$\mu$ is generating, it is also clear that the map
$\eta \mapsto \eta^{[\infty]}$ is injective.

To construct the inverse map, take an element
$x\in\vNN'\cap\vNN_{U}$.  Then $x_n = \E_n(x)$ is an element of
$(\vNN^{(n)})' \cap \vNN_{U}^{(n)}$.  Hence, for every $n\ge1$ and
$s\in\supp\mu^n$, there is a morphism $x_{n, s} \in \End{\CC}{U_s
\otimes U}$ such that $x_n$ is the direct sum of the $x_{n, s}$ (with
multiplicities). It follows that we can choose $\eta(n)\in \hat\CC(U)$
such that $\|\eta(n)\|\le\|x\|$ and $x_n=\eta(n)^{[n]}$. The elements
$\eta(n)$ are not uniquely determined, only their components
corresponding to $s\in\supp\mu^n$ are. The identity $\E_{n+1,
n}(x_{n+1}) = x_n$ translates into $P_\mu(\eta(n+1))_s=\eta(n)_s$ for
$s\in\supp\mu^n$.

We now define an element $\eta\in\hat\CC(U)$ by letting
$$
\eta_s=\eta(n)_s\ \ \text{if}\ \ s\in\supp\mu^n\ \ \text{for some}\ \
n\ge1.
$$
In order to see that this definition in unambiguous, assume
$s\in(\supp\mu^n)\cap(\supp\mu^{n+k})$ for some~$n$ and~$k$.  Then by
the $0$-$2$ law, see~\cite{MR2034922}*{Proposition~2.12}, we have
$\|P^m_\mu-P^{m+k}_\mu\|\to0$ as $m\to\infty$.  Since the sequence
$\{\eta(m)\}_m$ is bounded and we have
$\eta(n)_s=P^{m+k}_\mu(\eta(n+m+k))_s$ and
$\eta(n+k)_s=P^{m}_\mu(\eta(n+m+k))_s$, letting $m\to\infty$ we
conclude that $\eta(n)_s=\eta(n+k)_s$. Hence $\eta$ is well-defined,
$P_\mu$-harmonic, and $x_n=\eta^{[n]}$. Therefore $x=\eta^{[\infty]}$.

The linear isomorphism $\PP(U)\to\vNN'\cap\vNN_U$ and its inverse that
we have constructed, are unital and completely positive, hence they
are isomorphisms of von Neumann algebras.  \ep

As in the case of Proposition~\ref{prop:lr-emb}, the linear
isomorphism $\PP(U)\cong\vNN'\cap\vNN_U$ could be used to give an
alternative proof of Proposition~\ref{pproduct}, at least for
generating measures.

\smallskip

Applying Proposition~\ref{prop:p-bdry-as-rel-comm-LR0} to $U=\un$ we
get the following.

\begin{corollary}\label{cor:ergod-factor} The von Neumann algebra
$\vNN$ is a factor if and only if $\mu$ is ergodic.
\end{corollary}

Under a mildly stronger assumption on the measure we can prove a
better result than Proposition~\ref{prop:p-bdry-as-rel-comm-LR0},
which will be important later.

\begin{proposition}
\label{prop:p-bdry-as-rel-comm-LR} Assume that for any
$s,t\in\Irr(\CC)$ there exists $n\ge0$ such
that $$\supp(\mu^n*\delta_s)\cap\supp(\mu^n*\delta_t)\ne\emptyset.$$
Then for any objects $U$ and $U'$ in $\CC$, the map $\eta \mapsto
(\iota_{U'} \otimes \eta)^{[\infty]}$ defines an isomorphism of von
Neumann algebras $\End{\PP}{U}\cong \vNN_{U'}'\cap\vNN_{U' \otimes
U}$.
\end{proposition}

\bp That we get a map $\End{\PP}{U}\to \vNN_{U'}'\cap\vNN_{U' \otimes
U}$ does not require any assumptions on $\mu$ and is easy to see: if
$\eta$ is a harmonic element in $\hat\CC(U)$, the
naturality implies that the elements $\eta_{V^{\otimes m} \otimes U'}$
commute with $\End{\CC}{V^{\otimes n} \otimes U'}$ for $m \ge n$, and
hence $(\iota_{U'} \otimes \eta)^{[\infty]} = \lim_m \eta_{V^{\otimes
m} \otimes U'}$ lies in $ \vNN_{U'}'\cap\vNN_{U' \otimes U}$.

To construct the inverse map assume first $U'=U_t$ for some $t$. Take
$x\in \vNN_{U'}'\cap\vNN_{U' \otimes U}$. Similarly to the proof of
Proposition~\ref{prop:p-bdry-as-rel-comm-LR0} we can find elements
$\eta(n)\in\hat{\CC}(U)$ such that $\|\eta(n)\|\le\|x\|$ and
$\E_n(x)=(\iota_{U'}\otimes\eta(n))^{[n]}$.  The identity $\E_{n+1,
n}(x_{n+1}) = x_n$ means now that $P_\mu(\eta(n+1))_s=\eta(n)_s$ for
$s\in\supp(\mu^n*\delta_t)$. We want to define an element
$\eta\in\hat\CC(U)$ by
$$
\eta_s=\eta(n)_s\ \ \text{if}\ \ s\in\supp(\mu^n*\delta_t)\ \
\text{for some}\ \ n\ge1.
$$
As in the proof of Proposition~\ref{prop:p-bdry-as-rel-comm-LR0}, in
order to see that $\eta$ is well-defined, it suffices to show that if
$s\in\supp(\mu^n*\delta_t)\cap\supp(\mu^{n+k}*\delta_t)$ for some~$n$
and~$k$, then $\|P^m_\mu-P^{m+k}_\mu\|\to0$ as $m\to\infty$. Since
$\mu$ is assumed to be generating, there exists $l$ such that
$t\in\supp\mu^l$. But then
$$
s\in (\supp\mu^{n+l})\cap(\supp\mu^{n+l+k}),
$$
so the convergence $\|P^m_\mu-P^{m+k}_\mu\|\to0$ indeed holds by the
$0$-$2$ law. This finishes the proof of the proposition for $U'=U_t$,
and we see that no assumption in addition to the generating property
of $\mu$ is needed in this case.

\smallskip

Consider now an arbitrary $U'$. Decompose $U'$ into a direct sum of
simple objects:
$$
U' \cong U_{s_1}\oplus\dots\oplus U_{s_n}.
$$
Denote by $p_i\in\CC(U')$ the corresponding projections. Then the
inclusion $p_i\vNN_{U'}p_i\subset p_i\vNN_{U'\otimes U}p_i$ can be
identified with $\vNN_{U_{s_i}}\subset \vNN_{U_{s_i}\otimes U}$.

Take $x\in \vNN_{U'}'\cap\vNN_{U'\otimes U}$. Then $x$ commutes with
$p_i$. Since the element $xp_i$ lies in
$\vNN_{U_{s_i}}'\cap\vNN_{U_{s_i}\otimes U}$, it is defined by a
$P_\mu$-harmonic element $\eta(i)\in\hat{\CC}(U)$. In terms
of these elements the condition that $\E_n(x)$ commutes with
$\CC(V^{\otimes n}\otimes U')$ means that $\eta(i)_s=\eta(j)_s$
whenever $s\in\supp(\mu^n*\delta_{s_i})\cap\supp(\mu^n*\delta_{s_j})$,
while to finish the proof we need the equality $\eta(i)=\eta(j)$.

Fix $s\in\Irr(\CC)$ and indices $i$ and $j$. By assumption there
exists $t\in\supp(\mu^n*\delta_{s_i})\cap\supp(\mu^n*\delta_{s_j})$
for some $n$. Since $\mu$ is generating, there exists $m$ such that
$s\in\supp(\mu^m*\delta_t)$. Then
$$
s\in \supp(\mu^{m+n}*\delta_{s_i})\cap\supp(\mu^{m+n}*\delta_{s_j}),
$$
and therefore $\eta(i)_s=\eta(j)_s$.  \ep

Note that the proof shows that the additional assumption on the
measure is not only sufficient but also necessary for the result to
be true. Even for symmetric ergodic measures this condition does not
always hold: take the random walk on $\Z$ defined by the measure
$\mu=2^{-1}(\delta_{-1}+\delta_1)$. At the same time this condition is
satisfied, for example, for any generating measure $\mu$ with
$\mu(e)>0$. Indeed, for such a measure we can find $n$ such that
$s\in\supp(\mu^n*\delta_t)$, and then $s\in\supp(\mu^n*\delta_s)\cap
\supp(\mu^n*\delta_t)$.

\smallskip

Applying the proposition to $U=\un$ we get the following result.

\begin{corollary}\label{cor:engod-factor} Assume $\mu$ is ergodic and
satisfies the assumption of
Proposition~\ref{prop:p-bdry-as-rel-comm-LR}. Then $\vNN_U$ is a
factor for every object $U$ in $\CC$.
\end{corollary}

\begin{remark}\label{rmultiplicity} It is sometimes convenient to
consider slightly more general constructions allowing
multiplicities. Namely, instead of $V=\bigoplus_{s\in\supp\mu}U_s$ we
could take $V=\bigoplus_{i\in I}U_{s_i}$, where $(s_i)_{i\in I}$ is any
finite or countable collection of elements running through
$\supp\mu$. For the state on $\CC(V)$ we could take
$\CC(U_{s_i},U_{s_j})\ni T \mapsto\delta_{ij}\lambda_i\tr_{s_i}(T)$,
where $\lambda_i>0$ are any numbers such that $\sum_{i\colon
s_i=s}\lambda_i=\mu(s)$ for all $s\in\supp\mu$. All the above results
would remain true, with essentially identical proofs.
\end{remark}

\subsection{Relative commutants: Hayashi--Yamagami
  approach} \label{sHY}

We will now explain a modification of the Izumi--Longo--Roberts
construction due to Hayashi and Yamagami~\cite{MR1749868}. Its
advantage is that, at the expense of introducing an extra variable in
a II$_1$ factor, we can stay in the framework of finite von Neumann
algebras.

\smallskip

We continue to assume that $\mu$ is generating. We will use a slightly
different notation compared~\cite{MR1749868} to be more consistent
with the previous sections.

Let $\mathcal{R}$ be the hyperfinite II$_1$-factor and $\tau$ be the
unique normal tracial state on $\mathcal{R}$. Choose a partition of
unity by projections $(e_s)_{s \in \supp \mu}$ in $\RR$ which satisfy
$$
\tau(e_s) = \frac{\mu(s)}{c d(s)},\ \ \text{where}\ \
c=\sum_{s\in\supp\mu}\frac{\mu(s)}{ d(s)}.
$$
When $(s_n, \ldots, s_1) \in (\supp \mu)^n$, we write $e_{s_*} =
e_{s_n} \otimes \cdots \otimes e_{s_1} \in \mathcal{R}^{\otimes
n}$. As in Section~\ref{sec:longo-roberts-appr}, put
$V=\bigoplus_{s\in\supp\mu}U_s$. Now, for a fixed object $U$ in $\CC$,
instead of the algebra $\CC(V^{\otimes n}\otimes U)$ used there,
consider the algebra
$$
\tilde\CC(V^{\otimes n}\otimes U)=\bigoplus_{s_*, s'_* \in (\supp
\mu)^n} \CC(U_{s_n} \otimes \cdots \otimes U_{s_1} \otimes U, U_{s'_n}
\otimes \cdots \otimes U_{s'_1} \otimes U) \otimes e_{s'_*}
\mathcal{R}^{\otimes n} e_{s_*}.
$$
It carries a tracial state $\tau^{(n)}_U$ defined by
$$
\tau^{(n)}_U(T\otimes x)=\delta_{s_*,s'_*}c^nd(s_1)\dots
d(s_n)\tr_{U_{s_n} \otimes\dots\otimes U_{s_1}\otimes
U}(T)\tau^{\otimes n}(x)
$$
for $T\otimes x\in \CC(U_{s_n} \otimes \cdots \otimes U_{s_1} \otimes
U, U_{s'_n} \otimes \cdots \otimes U_{s'_1} \otimes U) \otimes
e_{s'_*} \mathcal{R}^{\otimes n} e_{s_*}$. Let $\A^{(n)}_U$ be the von
Neumann algebra generated by $\tilde\CC(V^{\otimes n}\otimes U)$ in
the GNS-representation defined by $\tau^{(n)}_U$. These algebras form
an inductive system under the embeddings
$$
\A^{(n)}_U\hookrightarrow\A^{(n+1)}_U,\ \ T\otimes x\mapsto
\sum_{s\in\supp\mu}(\iota_s\otimes T)\otimes(e_s\otimes x).
$$
Passing to the limit we get a von Neumann algebra $\A_U$ equipped with
a faithful tracial state $\tau_U$. We write $\A$ for $\A_\un$.

Given $\eta\in\hat{\CC}(U)$, consider the elements
$$
\eta^{\{n\}}=\sum_{s_*\in(\supp \mu)^n} \eta_{U_{s_n} \otimes \cdots
\otimes U_{s_1}}\otimes e_{s_*}\in\A^{(n)}_U.
$$
If $\eta$ is $P_\mu$-harmonic, then the sequence $\{\eta^{\{n\}}\}_n$
forms a martingale with respect to the $\tau_U$-preserving conditional
expectations $\E_n\colon\A_U\to\A^{(n)}_U$. Denote its limit by
$\eta^{\{\infty\}}$. Then we get the following analogues of
Propositions~\ref{prop:p-bdry-as-rel-comm-LR0}
and~\ref{prop:p-bdry-as-rel-comm-LR}, with almost identical proofs,
which we omit.

\begin{proposition}\label{prop:PP-equals-HY} For every object $U$ in
$\CC$, the map $\eta \mapsto\eta^{\{\infty\}}$ defines an isomorphism
of von Neumann algebras $\End{\PP}{U}\cong \A'\cap\A_{U}$. If in
addition to the generating property the measure $\mu$ satisfies the
assumption of Proposition~\ref{prop:p-bdry-as-rel-comm-LR}, then
the map $\eta \mapsto(\iota_{U'}\otimes\eta)^{\{\infty\}}$ also defines an
isomorphism of von Neumann algebras $\End{\PP}{U}\cong
\A_{U'}'\cap\A_{U'\otimes U}$ for any object $U'$.
\end{proposition}

The work of Hayashi and Yamagami contains much more than the
construction of the algebras $\A_U$ and, in fact, allows us to
describe, under mild additional assumptions on $\mu$, not only the
morphisms but the entire Poisson boundary $\Pi\colon\CC\to\PP$ in
terms of Hilbert bimodules over~$\A$.

For objects $X$ and $Y$ consider their direct sum $X\oplus Y$, and
denote by $p_X,p_Y\in\CC(X\oplus Y)$ the corresponding projections. We
can consider $p_X$ and $p_Y$ as projections in $\A_{X\oplus Y}$, then
$p_X(\A_{X\oplus Y})p_X \cong \A_X$ and $p_Y(\A_{X\oplus Y})p_Y \cong
\A_Y$. Put
$$
\A_{X,Y}=p_Y(\A_{X\oplus Y})p_X.
$$
The $\A_Y$-$\A_X$-module $\A_{X,Y}$ can be described as an inductive
limit of completions of the spaces
$$
\tilde\CC(V^{\otimes n}\otimes X,V^{\otimes n}\otimes
Y)=\bigoplus_{s_*, s'_* \in (\supp \mu)^n} \CC(U_{s_n} \otimes \cdots
\otimes U_{s_1} \otimes X, U_{s'_n} \otimes \cdots \otimes U_{s'_1}
\otimes Y) \otimes e_{s'_*} \mathcal{R}^{\otimes n} e_{s_*}.
$$
Denote by $\HH_X$ the Hilbert space completion of $\A_{\un,X}$ with
respect to the scalar product $$(x,y)=\tau_\un(y^*x).$$ Then $\HH_X$
is a Hilbert $\A_X$-$\A$-module (it is denoted by $X_\infty$
in~\cite{MR1749868}). Viewing $\HH_X$ as a Hilbert bimodule over $\A$,
we get a unitary functor $F$ from $\CC$ into the category $\Hilb_\A$
of Hilbert bimodules over~$\A$ such that $F(U)=\HH_U$ on objects and
defined in the obvious way on morphisms in $\CC$. We want to make~$F$ into a
tensor functor. By the computation on pp.~40--41 of~\cite{MR1749868}
the map
$$
\tilde\CC(V^{\otimes n},V^{\otimes n}\otimes X)\otimes
\tilde\CC(V^{\otimes n},V^{\otimes n}\otimes Y)\to\tilde
\CC(V^{\otimes n},V^{\otimes n}\otimes X\otimes Y),
$$
$$
(S \otimes a)\otimes (T\otimes b)\mapsto (S\otimes\iota_Y)T \otimes
ab,
$$
defines an isometry
$$
F_2(X,Y)\colon \HH_X\otimes_\A\HH_Y\to\HH_{X\otimes Y}.
$$

\begin{lemma} \label{lcondmeas} Assume that for every $s\in\Irr(\CC)$
we have
$$
(\mu^n*\delta_s)(\supp\mu^n)\to1\ \ \text{as}\ \ n\to\infty.
$$
Then the maps $F_2(X,Y)$ are unitary.
\end{lemma}

\bp It suffices to prove the lemma for simple objects. Assume $X=U_s$
for some $s$. For every $n\ge1$ and $s_*\in(\supp\mu)^n$, let
$p_{s_*}^{(n)}\in\CC(U_{s_n}\otimes\dots\otimes U_{s_1}\otimes X)$ be
the projection onto the direct sum of the isotypic components
corresponding to $U_t$ for some $t\in\supp\mu^n$. Put
$$
p^{(n)}=\sum_{s_* \in (\supp \mu)^n}p^{(n)}_{s_*}\otimes e_{s_*}\in\A^{(n)}_{U_s}.
$$
Then $\tau_X(p^{(n)})=(\mu^n*\delta_s)(\supp\mu^n)$. Therefore by
assumption $p^{(n)}\to1$ in the $s^*$-topology. It follows that, to
prove the lemma, it suffices to show that if
$$
T\otimes x\in \CC(U_{s_n} \otimes \cdots \otimes U_{s_1}, U_{s'_n}
\otimes \cdots \otimes U_{s'_1} \otimes X\otimes Y) \otimes e_{s'_*}
\mathcal{R}^{\otimes n} e_{s_*}
$$
is such that $p^{(n)}(T\otimes x)=T\otimes x$, then $T\otimes x$ is in
the image of $F_2(X,Y)$. The assumption on $T$ means that the simple
objects appearing in the decomposition of $U_{s'_n} \otimes \cdots
\otimes U_{s'_1} \otimes X$ appear also in the decomposition of
$U_{t_n}\otimes\dots\otimes U_{t_1}$ for $t_*\in(\supp\mu)^n$. This
implies that $T$ can be written as a finite direct sum of morphisms of
the form $(S\otimes\iota_Y)R$, with $R\in\CC(U_{s_n} \otimes \cdots
\otimes U_{s_1},U_{t_n}\otimes\dots\otimes U_{t_1}\otimes Y)$ and
$S\in\CC(U_{t_n} \otimes \cdots \otimes
U_{t_1},U_{s'_n}\otimes\dots\otimes U_{s'_1}\otimes X)$. Since we also
have density of $e_{s'_*}\RR e_{t_*}\RR e_{s_*}$ in $e_{s'_*}\RR
e_{s_*}$, this proves the lemma.  \ep

We remark that the assumption of the lemma is obviously satisfied if
$\supp\mu=\Irr(\CC)$. It is also satisfied if $\mu$ is ergodic and
$\mu(e)>0$, since then $\|\mu^n*\delta_s-\mu^n\|_1\to0$ by
\cite{MR1644299}*{Proposition~3.3}.

\smallskip

Once the maps $F_2(X,Y)$ are unitary, it is easy to see that $(F,F_2)$
is a unitary tensor functor $\CC\to\Hilb_\A$.

\begin{proposition} \label{pHYrealization} Assume the measure $\mu$
satisfies the assumption of Lemma~\ref{lcondmeas}. Let $\BB$ be the
full C$^*$-tensor subcategory of $\Hilb_\A$ generated by the image of
$F\colon\CC\to\Hilb_\A$. Then the Poisson boundary
$\Pi\colon\CC\to\PP$ of $(\CC,\mu)$ is isomorphic to
$F\colon\CC\to\BB$.
\end{proposition}

\bp The functor $F$ extends to the full subcategory $\tilde\PP$ of
$\PP$ formed by the objects of $\CC$ using the isomorphisms
$\PP(U)\cong\A'\cap\A_U$. It follows immediately by definition that
this way we get a unitary tensor functor $E\colon\tilde\PP\to\BB$ if
we put $E_2(X,Y)=F_2(X,Y)$. We then extend this functor to a unitary
tensor functor $\PP\to\BB$, which we continue to denote by $E$. To
prove the proposition it remains to show that $E$ is fully
faithful. In other words, we have to show that the left action of
$\A_U$ on $\HH_U$ defines an isomorphism
$\A'\cap\A_U\cong\Endd_{\A\mhyph\A}(\HH_U)$.

Let us check the stronger statement that the left action defines an
isomorphism $\A_U\cong\Endd_{\,\mhyph\A}(\HH_U)$. Recalling how $\HH_U$ was
constructed using complementary projections in $\A_{\un\oplus U}$, it
becomes clear that the map $\A_U\to\Endd_{\,\mhyph\A}(\HH_U)$ is always
surjective, and it is injective if and only if the projection
$p_\un\in\A_{\un\oplus U}$ has central support $1$. Using the
Frobenius reciprocity isomorphism
$$
\CC(V^{\otimes n}\otimes U)\cong \CC(V^{\otimes n},V^{\otimes
n}\otimes U\otimes\bar U),
$$
it is easy to check that $\HH_{U\otimes\bar U}\cong L^2(\A_U,\tau_U)$
as a Hilbert $\A_U$-$\A$-module. Hence the representation of~$\A_U$ on
$\HH_{U\otimes\bar U}$ is faithful. Since $\HH_{U\otimes\bar U}\cong
\HH_U\otimes_\A\HH_{\bar U}$, it follows that the representation of
$\A_U$ on~$\HH_U$ is faithful as well.  \ep

A similar result could also be proved using the algebras $\vNN_U$ from
Section~\ref{sec:longo-roberts-appr} instead of $\A_U$. The situation
would be marginally more complicated, since in dealing with the Connes
fusion tensor product~$\otimes_\vNN$ we would have to take into
account the modular group of $\omega_U$. We are not going to pursue
this topic here, although it could provide a somewhat alternative
route to Proposition~\ref{pminindex} below.

\bigskip

\section{A universal property of the Poisson
boundary}\label{suniversal}

Let $\CC$ be a weakly amenable strict C$^*$-tensor category. Fix an
ergodic probability measure $\mu$ on~$\Irr(\CC)$. Recall that such a
measure exists by Proposition~\ref{pweakamen}. Let
$\Pi\colon\CC\to\PP$ be the Poisson boundary of~$(\CC,\mu)$.

For an object $U$ in $\CC$ define
$$
\dmin^\CC(U)=\inf d^\A(F(U)),
$$
where the infimum is taken over all unitary tensor functors
$F\colon\CC\to\A$ from $\CC$ into rigid C$^*$-tensor categories $\A$. We will show in the next section that $\dmin^\CC$
is the amenable dimension function on $\CC$. The goal of the present
section is to prove the following.

\begin{theorem} \label{tuniversal1} The Poisson boundary
$\Pi\colon\CC\to\PP$ is a universal unitary tensor functor such that
$\dmin^\CC=d^\PP\Pi$.
\end{theorem}

In other words, $\dmin^\CC=d^\PP\Pi$ and for any unitary tensor
functor $F\colon\CC\to\A$ such that $\dmin^\CC=d^\A F$ there exists a
unique, up to a natural unitary monoidal isomorphism, unitary tensor
functor $\Lambda\colon\PP\to\A$ such that $\Lambda\Pi\cong F$.

\smallskip

For a rigid C$^*$-tensor category $\A$, consider a unitary tensor functor $F\colon\CC\to\A$, with no
restriction on the dimension function. As we discussed in
Section~\ref{sstensorfunctors}, we may assume that $\A$ is strict,
$\CC$ is a C$^*$-tensor subcategory of $\A$ and $F$ is the embedding
functor. Motivated by Izumi's Poisson integral~\cite{MR1916370} we
will define linear maps
$$
\Theta_{U,V}\colon\A(U,V)\to\PP(U,V).
$$
We will write $\Theta_U$ for $\Theta_{U,U}$ and often omit the
subscripts altogether, if there is no danger of confusion. The proof
of the theorem will be based on analysis of the multiplicative domain
of $\Theta$.

\smallskip

For every object $U$ in $\CC$ fix a standard solution $(R_U,\bar R_U)$
of the conjugate equations in $\CC$. Define a faithful state $\psi_U$
on $\End{\A}{U}$ by
$$
\psi_U(T)=d^\CC(X)^{-1} \bar R_U^*(T\otimes\iota)\bar R_U.
$$
Since any other standard solution has the form
$((u\otimes\iota)R_U,(\iota\otimes u)\bar R_U)$ for a unitary $u$,
this definition is independent of any choices. More generally, we can
define in a similar way ``slice maps"
$$
\iota\otimes\psi_V\colon \End{\A}{U\otimes V}\to \End{\A}{U}.
$$
Then, since $((\iota\otimes R_U\otimes \iota)R_V,(\iota\otimes \bar
R_V\otimes \iota)\bar R_U)$ is a standard solution for $U\otimes V$,
we get
\begin{equation} \label{eslice1} \psi_{U\otimes V}=\psi_U(\iota\otimes
\psi_V).
\end{equation}

By definition the state $\psi_U$ extends the trace $\tr_U$ on
$\End{\CC}{U}$.

\begin{lemma} The subalgebra $\End{\CC}{U}\subset \End{\A}{U}$ is
contained in the centralizer of the state $\psi_U$.
\end{lemma}

\bp If $u$ is a unitary in $\CC(U)$, then the state $\psi_U(u\cdot
u^*)$ is defined similarly to $\psi_U$, but using the solution
$((\iota\otimes u^*)R_U,(u^*\otimes\iota)\bar R_U)$ of the conjugate
equations for $U$. Since $\psi_U$ is independent of the choice of
standard solutions, it follows that $\psi_U(u\cdot u^*)=\psi_U$. But
this exactly means that $\CC(U)$ is contained in the centralizer of
$\psi_U$.  \ep

It follows that there exists a unique $\psi_U$-preserving conditional
expectation $E_U\colon \End{\A}{U}\to \End{\CC}{U}$.
For objects $U$ and $V$ we can consider $\A(U,V)$ as a subspace of
$\End{\A}{U\oplus V}$. Then $E_{U\oplus V}$ defines a linear map
$$
E_{U,V}\colon\A(U,V)\to\CC(U,V).
$$
Again, we omit the subscripts when convenient.

\begin{lemma} \label{lcondexp1} The maps $E_{U,V}$ satisfy the
following properties:
\begin{enumerate}
\item $E_{U,V}(T)^*=E_{V,U}(T^*)$;
\item if $T\in\A(U,V)$ and $S\in\CC(V,W)$, then
$E_{U,W}(ST)=SE_{U,V}(T)$;
\item for any object $X$ in $\CC$ we have $E_{U\otimes
X,V\otimes X}(T\otimes\iota_X)=E_{U,V}(T)\otimes\iota_X$.
\end{enumerate}
\end{lemma}

\bp Properties (i) and (ii) follows immediately from the corresponding
properties of conditional expectations. To prove (iii), it suffices to
consider the case $U=V$. Take $S\in\End{\CC}{U\otimes X}$. Then we
have to check that
$$
\psi_{U\otimes X}(S(T\otimes\iota))=\psi_{U\otimes X}(S(E(T)\otimes
\iota)).
$$
This follows from \eqref{eslice1} and the fact that by definition we
have $(\iota\otimes\psi_X)(S)\in\End{\CC}{U}$.  \ep

Now, given a morphism $T\in \A(U,V)$, define a bounded natural
transformation $\Theta_{U,V}(T)\colon\iota\otimes U\to\iota\otimes V$
of functors on $\CC$ by
$$
\Theta_{U,V}(T)_X=E_{X\otimes U,X\otimes V}(\iota_X\otimes T).
$$

\begin{lemma}\label{luniharmonic} The natural transformation
$\Theta_{U,V}(T)$ is $P_X$-harmonic for any object $X$ in $\CC$.
\end{lemma}

\bp It suffices to consider the case $U=V$. We claim that
$$
(\tr_X\otimes \iota)E(\iota_X\otimes T)=E(T).
$$
Indeed, for any $S\in \End{\CC}{U}$ we have
$$
\tr_U\big(S(\tr_X\otimes \iota)E(\iota_X\otimes T)\big) =\tr_{X\otimes
U}(E(\iota_X\otimes ST))\\ =\psi_{X\otimes U}(\iota\otimes ST)
=\psi_U(ST)=\tr_U(SE(T)),
$$
where in the third equality we used \eqref{eslice1}. This proves the
claim.

We now compute:
$$
P_X(\Theta(T))_Y=(\tr_X\otimes\iota)(\Theta(T)_{X\otimes
Y})=(\tr_X\otimes\iota_{Y}\otimes\iota_U)(E(\iota_{X}\otimes\iota_Y\otimes
T))\\ =E(\iota_Y\otimes T)=\Theta(T)_Y,
$$
so $\Theta(T)$ is $P_X$-harmonic.  \ep

It follows that $\Theta_{U,V}$ is a well-defined linear map
$\A(U,V)\to \PP(U,V)$.

\begin{lemma} \label{ltheta} The maps $\Theta_{U,V}$ satisfy the
following properties:
\begin{itemize}
\item[{\rm (i)}] $\Theta_{U,V}(T)^*=\Theta_{V,U}(T^*)$;
\item[{\rm (ii)}] if $T\in\A(U,V)$ and $S\in\CC(V,W)$, then
$\Theta_{U,W}(ST)=S\Theta_{U,V}(T)$;
\item[{\rm (iii)}] for any object $X$ in $\CC$ we have
$\Theta_{U\otimes X,V\otimes
X}(T\otimes\iota_X)=\Theta_{U,V}(T)\otimes\iota_X$ and
$\Theta_{X\otimes U,X\otimes V}(\iota_X\otimes T)=\iota_X\otimes
\Theta_{U,V}(T)$;
\item[{\rm (iv)}] the maps $\Theta_U\colon\End{\A}{U}\to\End{\PP}{U}$
are unital, completely positive, and faithful.
\end{itemize}
\end{lemma}

\bp All these properties are immediate consequences of the definitions
and the properties of the maps $E_{U,V}$ given in
Lemma~\ref{lcondexp1}. We would like only to point out that the
property $\Theta(\iota\otimes T)=\iota\otimes\Theta(T)$ follows from
the definition of the tensor product in $\PP$, the corresponding
property for the maps $E$ is neither satisfied nor needed.  \ep

Our goal now is to understand the multiplicative domains of the maps
$\Theta_U\colon\End{\A}{U}\to\End{\PP}{U}$. We will first show that
these domains cannot be very large. More precisely, assume we have an
intermediate C$^*$-tensor category $\CC\subset\BB\subset\A$ such that
$d^\A=d^\BB$ on $\CC$. For an object $U$ in $\CC$ denote by
$E^\BB_U\colon\End{\A}{U}\to\BB(U)$ the conditional expectation
preserving the categorical trace on $\A$. Then we have the following
result inspired by \cite{MR2335776}*{Lemma~4.5}.

\begin{lemma} \label{lthetafactor} We have $\Theta_U=\Theta_U
E^\BB_U$.
\end{lemma}

\bp We will first show that a similar property holds for the maps $E$,
so $E_U=E_UE^\BB_U$.

Consider the normalized categorical trace $\tr^\A_U$ on
$\End{\A}{U}$. We have $\psi_U=\tr^\A_U(\cdot\, Q)$ for some
$Q\in\End{\A}{U}$. The identity $E_U=E_UE^\BB_U$ holds if and only if
the conditional expectation $E^\BB_U$ is $\psi_U$-preserving, or
equivalently, $Q\in\BB(U)$.

By assumption we have $d^\A(U)=d^\BB(U)$ for every object $U$ in
$\CC$. It follows that a standard solution $(R^\BB_U,\bar R^\BB_U)$ of
the conjugate equations for $U$ and $\bar U$ in $\BB$ remains standard
in $\A$. We have $\bar R_U=(T\otimes\iota)\bar R^\BB_U$ for a uniquely
defined $T\in\BB(U)$. Then $Q=\frac{d^\BB(U)}{d^\CC(U)}TT^*\in\BB(U)$.

\smallskip

We also need the simple property $E^\BB_{X\otimes U}(\iota_X\otimes
T)=\iota_X\otimes E^\BB_U(T)$. This is proved similarly to
Lemma~\ref{lcondexp1}(iii), using that $\tr^\A_{X\otimes
U}=\tr^\A_U(\tr^\A_X\otimes\iota)$ and the fact that $\tr^\A$ is
defined using standard solutions in $\BB$, so that
$(\tr^\A_X\otimes\iota)(\BB(X\otimes U))\subset\BB(U)$.

\smallskip

The equality $\Theta_U E^\BB_U=\Theta_U$ is now immediate:
$$
\Theta E^\BB(T)_X=E(\iota_X\otimes E^\BB(T))=E E^\BB(\iota_X\otimes T)
=E(\iota_X\otimes T)=\Theta(T)_X.
$$
This proves the assertion.
\ep

Since the completely positive map $\Theta_U$ is faithful, the
multiplicative domain of $\Theta_U=\Theta_UE^\BB_U$ is contained in
that of $E^\BB_U$, which is exactly $\BB(U)$. Therefore to find this
domain we have to consider the smallest possible subcategory that
contains $\CC$ and still defines the same dimension function as~$\A$.

\begin{lemma} \label{luniversalau} For every object $U$ in $\CC$ there
exists a unique positive invertible element $a_U\in{\A}(U)$ such that
$$
(\iota\otimes a_U^{1/2})R_U\ \ \text{and}\ \
(a^{-1/2}_U\otimes\iota)\bar R_U
$$
form a standard solution of the conjugate equations for $U$ in $\A$.
\end{lemma}

\bp We can find an invertible element $T\in\End{\A}{U}$ such that
$(\iota\otimes T)R_U$ and $((T^*)^{-1}\otimes\iota)\bar R_U$ form a
standard solution in $\A$. Then we can take $a_U=T^*T$, since
$Ta_U^{-1/2}$ is unitary and hence the morphisms $(\iota\otimes
a_U^{1/2})R_U$ and $(a^{-1/2}_U\otimes\iota)\bar R_U$ still form a
standard solution.

Any other standard solution for $U$ and $\bar U$ has the form
$(\iota\otimes va_U^{1/2})R_U$, $(va^{-1/2}_U\otimes\iota)\bar R_U$
for a unitary $v\in\End{\A}{U}$. By uniqueness of the polar
decomposition the element $va_U^{1/2}$ is positive only if $v=1$.  \ep

Note that if we replace $(R_U,\bar R_U)$ by $((\iota\otimes
u)R_U,(u\otimes\iota)\bar R_U)$ for a unitary $u\in\End{\CC}{U}$, then
$a_U$ gets replaced by $ua_Uu^*$.

\begin{lemma} \label{luniversaldim} For every object $U$ in $\CC$ we
have $d^\PP(U)\le d^\A(U)$, and if the equality holds, then we have
$\Theta_U(a_U)^{-1}=\Theta_U(a_U^{-1})$.
\end{lemma}

\bp As usual, we omit the subscript $U$ in the computations. Consider
the solution
\begin{align*}
r&=(\iota\otimes \Theta(a)^{1/2})R,&
\bar{r}&=(\Theta(a)^{-1/2}\otimes\iota)\bar R
\end{align*}
of the conjugate equations for $U$ in $\PP$. Then from the equality
$$
r^*r=R^*(\iota\otimes \Theta(a))R=\Theta(R^*(\iota\otimes a)R),
$$
we have $\|r\|=d^{\A}(U)^{1/2}$. On the other hand, we also have
$$
\bar r^*\bar r=\bar R^*(\Theta(a)^{-1}\otimes \iota)\bar R.
$$
By Jensen's inequality for positive maps and the fact that the
function $t\mapsto t^{-1}$ on $(0,+\infty)$ is operator convex (see,
e.g.,~\cite{MR2251116}*{B.2}), we have
$\Theta(a)^{-1}\le\Theta(a^{-1})$. Hence we have the estimate
$$
\bar r^*\bar r\le\bar R^*(\Theta(a^{-1})\otimes \iota)\bar
R=\Theta(\bar R^*(a^{-1}\otimes \iota)\bar R),
$$
and we conclude that $\|\bar r\|\le d^{\A}(U)^{1/2}$. Hence
$d^\PP(U)\le d^\A(U)$, and if the equality holds, then we have $\|\bar
r\|= d^{\A}(U)^{1/2}$ and
$$
\bar R^*(\Theta(a)^{-1}\otimes \iota)\bar R=\bar
R^*(\Theta(a^{-1})\otimes \iota)\bar R.
$$
Since $T\mapsto \bar R^*(T\otimes \iota)\bar R$ is a faithful positive
linear functional on $\End{\PP}{U}$, this is equivalent to
$\Theta(a)^{-1}=\Theta(a^{-1})$.  \ep

If we have $d^\PP(U)=d^\A(U)$, we can then apply the following general result,
which is surely well-known.

\begin{lemma} \label{lmultdomain} Assume $\theta\colon A\to B$ is a
unital completely positive map of C$^*$-algebras and $a\in A$ is a
positive invertible element such that
$\theta(a)^{-1}=\theta(a^{-1})$. Then $a$ lies in the multiplicative
domain of $\theta$.
\end{lemma}

\bp It suffices to show that $a^{1/2}$ lies in the multiplicative
domain. This, in turn, is equivalent to the equality
$\theta(a)^{1/2}=\theta(a^{1/2})$.

Using Jensen's inequality and operator convexity of the functions
$t\mapsto -t^{1/2}$ and $t\mapsto t^{-1}$, we have
$$
\theta(a)^{1/2}\ge\theta(a^{1/2}),\ \
\theta(a^{-1})^{1/2}\ge\theta(a^{-1/2}),\ \ \text{and}\ \
\theta(a^{-1/2})^{-1}\le\theta(a^{1/2}).
$$
The second and the third inequalities imply
$$
\theta(a^{-1})^{-1/2}\le \theta(a^{1/2}).
$$
Since $\theta(a^{-1})=\theta(a)^{-1}$, this gives $\theta(a)^{1/2}\le
\theta(a^{1/2})$. Hence $\theta(a)^{1/2}=\theta(a^{1/2})$.  \ep

To finish the preparation for the proof of Theorem~\ref{tuniversal1}
we consider the maps $\Theta$ for $\A=\PP$.

\begin{lemma} \label{lthetaid} The maps
$\Theta_{U,V}\colon\PP(U,V)\mapsto\PP(U,V)$ defined by the functor
$\Pi\colon\CC\to\PP$ are the identity maps.
\end{lemma}

\bp It suffices to consider $U=V$. Take $\eta\in\End{\PP}{U}$. Let us
show first that $\Theta(\eta)_\un=\eta_\un$, that is,
$E(\eta)=\eta_\un$. In other words, we have to check that for any
$S\in\End{\CC}{U}$ we have
$$
\psi_U(S\eta)=\tr_U(S\eta_\un).
$$
This follows immediately by definition, since
$$
(\bar R_U^*(S\eta\otimes\iota)\bar R_U)_\un=\bar
R_U^*(S\eta_\un\otimes\iota)\bar R_U.
$$

Now, for any object $X$ in $\CC$, we have
$$
\Theta(\eta)_X=E(\iota_X\otimes\eta)=\Theta(\iota_X\otimes\eta)_\un=(\iota_X\otimes\eta)_\un=\eta_X.
$$
Therefore we have $\Theta(\eta)=\eta$.  \ep

\bp[Proof of Theorem~\ref{tuniversal1}] The equality
$\dmin^\CC(U)=d^\PP(U)$ for objects $U$ in $\CC$ follows from
Lemma~\ref{luniversaldim}.

\smallskip

Let $F\colon\CC\to\A$ be a unitary tensor functor such that
$\dmin^\CC=d^\A F$. As above, we assume that $F$ is simply an
embedding functor. Consider the minimal subcategory
$\tilde\BB\subset\A$ containing $\CC\subset\A$ and the
morphisms~$\iota_V\otimes a_U\otimes\iota_W$ for all objects $V$, $U$
and $W$ in $\CC$, where $a_U\in\End{\A}{U}$ are the morphisms defined
in Lemma~\ref{luniversalau}. This is a C$^*$-tensor subcategory, in
general without subobjects. Complete~$\tilde\BB$ with respect to
subobjects to get a C$^*$-tensor category~$\BB$. By adding more
objects to $\A$ we may assume without loss of generality that
$\BB\subset\A$. Lemmas~\ref{ltheta},~\ref{luniversaldim}
and~\ref{lmultdomain} imply that the maps $\Theta_{U,V}$ define a
strict unitary tensor functor $\tilde\BB\to\PP$. Thus $\tilde\BB$ is
unitary monoidally equivalent to a C$^*$-tensor subcategory
$\tilde\PP\subset \PP$, possibly without subobjects. Completing
$\tilde\PP$ with respect to subobjects we get a C$^*$-tensor
subcategory $\PP'\subset\PP$, which is unitarily monoidally equivalent
to $\BB$.

We claim that the embedding functor $\PP'\to\PP$ is a unitary monoidal
equivalence. Indeed, by construction we have
$d^{\PP'}(U)=\dmin^\CC(U)$ for every object $U$ in $\CC$. By
Lemmas~\ref{lthetafactor} and~\ref{lthetaid} it follows then that the
identity maps $\End{\PP}{U}\to\End{\PP}{U}$ factor through the
conditional expectations
$E^{\PP'}_U\colon\End{\PP}{U}\to\End{\PP'}{U}$. Hence
$\End{\PP}{U}=\End{\PP'}{U}$. Since the objects of $\CC$ generate
$\PP$, this implies that the embedding functor $\PP'\to\PP$ is a
unitary monoidal equivalence.

We have therefore shown that $\PP$ and $\BB$ are unitarily monoidally
equivalent, and furthermore, by properties of the maps $\Theta$ such
an equivalence $\Lambda\colon\PP\to\BB$ can be chosen to be the
identity tensor functor on $\CC$. Considered as a functor $\PP\to\A$,
the unitary tensor functor $\Lambda$ gives the required factorization
of $F\colon\CC\to\A$.

\smallskip

It remains to prove uniqueness. Denote by $\rho_U\in\End{\PP}{U}$ the
elements $a_U$ constructed in Lemma~\ref{luniversalau} for the
category $\PP$. By the uniqueness part of that lemma, it is clear that
any unitary tensor functor $\Lambda\colon\PP\to\A$ extending the
embedding functor $\CC\to\A$ must map $\rho_U\in\End{\PP}{U}$ into
$a_U\in\End{\A}{U}$. But this completely determines $\Lambda$ up to a
unitary monoidal equivalence, since by the above considerations the
category $\PP$ is obtained from $\CC$ by adding the morphisms $\rho_U$
and then completing the new category with respect to subobjects.  \ep

We finish the section with a couple of corollaries.

\smallskip

The universality of the Poisson boundary implies that up to an
isomorphism the boundary does not depend on the choice of an ergodic
measure. But the proof shows that a stronger result is true.

\begin{corollary} \label{cindepend} Let $\CC$ be a weakly amenable
C$^*$-tensor category and $\mu$ be an ergodic probability measure on
$\Irr(\CC)$. Then any bounded $P_\mu$-harmonic natural transformation
is $P_s$-harmonic for every $s\in\Irr(\CC)$, so the Poisson boundary
$\Pi\colon\CC\to\PP$ of $(\CC,\mu)$ does not depend on the choice of
an ergodic measure.
\end{corollary}

\bp By Lemma~\ref{lthetaid} the maps
$\Theta_{U,V}\colon\PP(U,V)\to\PP(U,V)$ are the identity maps, while
by Lem\-ma~\ref{luniharmonic} their images consist of elements that
are $P_s$-harmonic for all $s$.  \ep

When $\CC$ is amenable, then $\dmin^\CC=d^\CC$ and we get the
following.

\begin{corollary} \label{camenb} If $\CC$ is an amenable C$^*$-tensor
category, then its Poisson boundary with respect to any ergodic
probability measure on $\Irr(\CC)$ is trivial. In other words, any
bounded natural transformation $\iota\otimes U\to\iota\otimes V$ which
is $P_s$-harmonic for all $s\in\Irr(\CC)$, is defined by a morphism
in~$\CC(U,V)$.
\end{corollary}

\bp The identity functor $\CC\to\CC$ is already universal, so it is
isomorphic to the Poisson boundary.  \ep

We remark that if we were interested only in proving this corollary, a
majority of the above arguments, being applied to the functor
$\CC\to\PP$, would become either trivial or unnecessary. Namely, in
this case a standard solution of the conjugate equations in~$\CC$
remains standard in $\PP$, so we have $E_U=E^\CC_U$, and the key parts
of the proof are contained in Lemmas~\ref{lthetafactor}
and~\ref{lthetaid}. The first lemma shows that given
$\eta\in\End{\PP}{U}$ we have $E(\iota_X\otimes\eta)=E(\iota_X\otimes
E(\eta))$, while the second shows that
$E(\iota_X\otimes\eta)=\eta_X$. Since $E(\iota_X\otimes
E(\eta))=\iota_X\otimes E(\eta)$, we therefore see that $\eta$
coincides with $E(\eta)\in\End{\CC}{U}$.

\smallskip

Corollary~\ref{camenb} is more or less known: in view of
Proposition~\ref{prop:PP-equals-HY}, for measures considered
in~\cite{MR1749868} it is equivalent
to~\cite{MR1749868}*{Theorem~7.6}. For an even more restrictive class
of measures the result also follows
from~\cite{MR1444286}*{Theorem~5.16}.

\bigskip

\section{Amenability of the minimal dimension function}
\label{sec:amen-minim-dimens}

As in the previous section, let $\CC$ be a weakly amenable strict
C$^*$-tensor category. We defined the dimension function $\dmin^\CC$
on $\CC$ as the infimum of dimension functions under all possible
embeddings of $\CC$ into rigid C$^*$-tensor categories, and showed that it
is indeed a dimension function realized by the Poisson boundary of
$\CC$ with respect to any ergodic measure. The goal of this section is
to prove the following.

\begin{theorem} \label{tminimal} The dimension function $\dmin^\CC$ is
  amenable, that is, $\dmin^\CC(U)=\|\Gamma_U\|$ holds for every
  object $U$ in $\CC$.
\end{theorem}

We remark that already the fact that the fusion algebra of a weakly
amenable C$^*$-tensor category admits an amenable dimension function
is nontrivial. We do not know whether this is true for weakly amenable
dimension functions on fusion algebras that are not of categorical
origin. If the fusion algebra is commutative, this is true by a result
of Yamagami~\cite{MR1721584}.

\smallskip

Let $\mu$ be an ergodic probability measure $\mu$ on $\Irr(\CC)$ and
consider the corresponding Poisson boundary $\Pi\colon\CC\to\PP$. By
Theorem~\ref{tuniversal1} we already know that
$\dmin^\CC=d^\PP\Pi$. Therefore Theorem~\ref{tminimal} is equivalent
to saying that $d^\PP\Pi$ is the amenable dimension function on
$\CC$.

\smallskip

We will use the realization of harmonic transformations as elements of
$\vNN'\cap\vNN_U$ given in Section~\ref{sec:longo-roberts-appr}. It
will also be important to work with factors. Therefore we assume that
in addition to being ergodic the measure $\mu$ is generating and
satisfies the assumption of
Proposition~\ref{prop:p-bdry-as-rel-comm-LR} (recall that for the
latter it suffices to require $\mu(e)>0$). Recall once again that by
Proposition~\ref{pweakamen} such a measure exists. We also remind that
by Corollary~\ref{cindepend} the Poisson boundary does not depend on
the ergodic measure, but its realization in terms of relative
commutants does. We then have the following expected (in view of
Proposition~\ref{pHYrealization} and the discussion following it), but
crucial, result.

\begin{proposition} \label{pminindex} For every object $U$ in $\CC$,
we have $d^\PP(U)=[\vNN_U\colon\vNN]_0^{1/2}$, where
$[\vNN_U\colon\vNN]_0$ is the minimal index of the subfactor
$\vNN\subset\vNN_U$.
\end{proposition}

Before we turn to the proof, recall the construction of
$\vNN_U$. Consider $V=\bigoplus_{s\in\supp\mu}U_s$. We will work with $V$
as with a well-defined object. If $\supp\mu$ is infinite, to be
rigorous, in what follows we have to replace $V$ by finite sums of
objects $U_s$, $s\in\supp\mu$, and then pass to the limit, but we will
omit this repetitive simple argument. With this understanding,
$\vNN_U$ is the inductive limit of the algebras
$\vNN^{(n)}_U=\CC(V^{\otimes n}\otimes U)$ equipped with the faithful
states $\omega^{(n)}_U$.

Given another object $U'$, the partial trace $\iota \otimes \tr_{U}$
defines, for each $n$, a conditional expectation $\vNN_{U' \otimes
U}^{(n)} \to \vNN_{U'}^{(n)}$ which preserves the state
$\omega^{(n)}_{U'\otimes U}$.  The conditional expectation $\vNN_{U'
\otimes U} \to \vNN_{U'}$ which we get in the limit, is denoted by
$E_{U',U}$, or simply by $E_U$ if there is no danger of confusion. Fix
a standard solution $(R_U,\bar R_U)$ of the conjugate equations for
$U$ in $\CC$.

\begin{lemma}\label{l:basic-ext} The index of the conditional
expectation $E_U\colon\vNN_U\to\vNN$ equals $d^\CC(U)^2$, the
corresponding basic extension is $\vNN_U\subset\vNN_{U \otimes
\bar{U}}$, with the Jones projection $e_U= d^\CC(U)^{-1} \bar{R}_U
\bar{R}_U^* \in\vNN^{(0)}_{U\otimes\bar U}\subset \vNN_{U \otimes
\bar{U}}$ and the conditional expectation
$E_{\bar{U}}\colon\vNN_{U\otimes\bar U}\to\vNN_U$.
\end{lemma}

\bp By the abstract characterization of the basic extension
\cite{MR1245827}*{Theorem~8} it suffices to check the following three
properties: $E_{\bar U}(e_U)=d^\CC(U)^{-2}1$, $E_{\bar
U}(xe_U)e_U=d^\CC(U)^{-2}xe_U$ for all $x\in\vNN_{U\otimes\bar U}$,
and $e_Uxe_U=E_U(x)e_U$ for all $x\in\vNN_U$. The first and the third
properties are immediate by definition. To prove the second, it is
enough to show that for all $x\in\CC(X\otimes U\otimes\bar U)$ we have
$$
d^\CC(U)\big((\iota\otimes\tr_{\bar U})(x(\iota_X\otimes \bar R_U\bar
R^*_U))\otimes\iota_{\bar U}\big) (\iota_X\otimes\bar R_U\bar
R_U^*)=x(\iota_X\otimes\bar R_U\bar R_U^*).
$$
The left hand side equals
\begin{gather*}
  (\iota_{X}\otimes\iota_{U}\otimes R^*_U\otimes\iota_{\bar U})
  (x\otimes\iota_{U}\otimes\iota_{\bar U}) (\iota_X\otimes \bar
  R_U\bar R^*_U\otimes\iota_{U}\otimes\iota_{\bar U})
  (\iota_{X}\otimes\iota_{U}\otimes R_U\otimes\iota_{\bar U})
  (\iota_X\otimes\bar R_U\bar R_U^*)\\
\begin{split}
 &= (\iota_{X}\otimes\iota_{U}\otimes R^*_U\otimes\iota_{\bar U})
 (x\otimes\iota_{U}\otimes\iota_{\bar U}) (\iota_X\otimes \bar
 R_U\otimes\iota_{U}\otimes\iota_{\bar U}) (\iota_X\otimes\bar R_U\bar R_U^*)\\
 &= (\iota_{X}\otimes\iota_{U}\otimes R^*_U\otimes\iota_{\bar U})
 (x\otimes\iota_{U}\otimes\iota_{\bar U}) (\iota_X\otimes \bar
 R_U\otimes\bar R_U) (\iota_X\otimes\bar R_U^*)\\
 &=x(\iota_X\otimes\bar R_U\bar R_U^*),
\end{split}
\end{gather*}
which proves the lemma.
\end{proof}

This lemma implies in particular that there exists a unique
representation $$\pi\colon\vNN_{U\otimes\bar U}\to
B(L^2(\vNN_U,\omega_U))$$ that extends the representation of $\vNN_U$
and is such that $\pi(e_U)$ is the projection onto the closure of
$\Lambda_{\omega_U}(\vNN)\subset L^2(\vNN_U,\omega_U)$.

\begin{lemma} \label{lpi} The representation
  $\pi\colon\vNN_{U\otimes\bar U}\to B(L^2(\vNN_U,\omega_U))$ is given
  by
$$
\pi(x)\Lambda_{\omega_U}(y) = \Lambda_{\omega_U}((\iota\otimes
R_U^*)(x\otimes\iota_U)(y\otimes\iota_{\bar U}\otimes
\iota_U)(\iota\otimes \bar R_U\otimes\iota_U))
$$
for $x\in\bigcup_n\vNN^{(n)}_{U\otimes\bar U}$ and
$y\in\bigcup_n\vNN^{(n)}_U$.
\end{lemma}

\bp Let us write $\tilde\pi(x)$ for the operators in the formulation
of the lemma. The origin of the formula for $\tilde\pi$ is the
Frobenius reciprocity isomorphism
$$
\CC(V^{\otimes n}\otimes U) \cong \CC(V^{\otimes n}, V^{\otimes
  n}\otimes U\otimes \bar U),\ \ T\mapsto (T\otimes\iota_{\bar
  U})(\iota_{V^{\otimes n}}\otimes\bar R_U),
$$
with inverse $S\mapsto (\iota\otimes R^*_U)(S\otimes\iota_{U})$. Up to
scalar factors these isomorphisms become unitary once we equip both
spaces with scalar products defined by the states $\omega^{(n)}_U$ and
$\omega^{(n)}_\un$, respectively. The algebra $\CC(V^{\otimes
  n}\otimes U\otimes \bar U)$ is represented on $ \CC(V^{\otimes
  n},V^{\otimes n}\otimes U\otimes \bar U)$ by the operators of
multiplication on the left. Being written on the space $\CC(V^{\otimes
  n}\otimes U)$, this representation is exactly $\tilde\pi$. Therefore
$\tilde\pi$ certainly defines a representation of the $*$-algebra
$\bigcup_n\vNN^{(n)}_{U\otimes\bar U}$ on the dense subspace $\bigcup_n
L^2(\vNN^{(n)}_U,\omega^{(n)}_U)$ of $L^2(\vNN_U,\omega_U)$. In order
to see that this representation extends to a normal representation of
$\vNN_{U\otimes\bar U}$, observe that the vector
$\Lambda_{\omega_U}(1)$ is cyclic and
$$
\left(\tilde\pi(x) \Lambda_{\omega_U}(1), \Lambda_{\omega_U}(1)\right)
= d^\CC(U)^2 \omega_{U\otimes\bar U}(e_{U} x e_{U}),
$$
since for every $z\in\CC(U\otimes\bar U)$ we have
\begin{align*} \tr_U((\iota_U\otimes R_U^*)(z\otimes\iota_U)(\bar
R_U\otimes\iota_U)) &=d^\CC(U)^{-1}\bar R_U^*(\iota_U\otimes
R_U^*\otimes\iota_{\bar U})(z\otimes\iota_U\otimes\iota_{\bar U})(\bar
R_U\otimes\iota_U\otimes\iota_{\bar U})\bar R_U\\ &=d^\CC(U)^{-1}\bar
R_U^*z\bar R_U=d^\CC(U)\tr_{U\otimes\bar U}(z\bar R_U^*\bar R_U)\\
&=d^\CC(U)^2\tr_{U\otimes\bar U}(ze_U)=d^\CC(U)^2\tr_{U\otimes\bar
U}(e_Uze_U).
\end{align*}

It is clear that
$\tilde\pi(x)\Lambda_{\omega_U}(y)=\Lambda_{\omega_U}(xy)$ for
$x\in\bigcup_n\vNN^{(n)}_U$, so $\tilde\pi$ extends the representation of
$\vNN_U$ on $L^2(\vNN_U,\omega_U)$. Therefore to prove that
$\pi=\tilde\pi$ it remains to show that $\tilde\pi(e_U)$ is the
projection onto $\overline{\Lambda_{\omega_U}(\vNN)}$, that~is,
$$
\tilde\pi(e_U)\Lambda_{\omega_U}(y)=\Lambda_{\omega_U}(E_U(y))\ \
\text{for}\ \ y\in\bigcup_n\vNN^{(n)}_U.
$$
But this is obvious, as $(\iota_U\otimes R^*_U)(\bar R_U\bar
R^*_U\otimes\iota_U)=\bar R^*_U\otimes\iota_U$.  \ep

It is easy to describe the modular group $\sigma^{\omega_U}$ of
$\omega_U$. For $s_* = (s_1, \ldots, s_n) \in (\supp \mu)^n$, let us
put
$$
\delta_{s_*} = \frac{\mu(s_1) \cdots \mu(s_n)}{d^\CC(U_{s_1}) \cdots
d^\CC(U_{s_n})}.
$$
Then
$$
\sigma^{\omega_U}_t(x)=\left(\frac{\delta_{s'_*}}{\delta_{s_*}}\right)^{it}x\
\ \text{for}\ \ x\in \CC(U_{s_n} \otimes \cdots \otimes U_{s_n}\otimes
U, U_{s'_n} \otimes \cdots \otimes U_{s'_1}\otimes U).
$$
What matters for us is that since the automorphisms
$\sigma^{\omega_U}_t$ are approximately implemented by unitaries in
$\vNN$, the relative commutant $\vNN'\cap\vNN_U$ is contained in the
centralizer of the state $\omega_U$.

\smallskip

Consider the modular conjugation $J=J_{\omega_U}$ on
$L^2(\vNN_U,\omega_U)$. By Lemma~\ref{l:basic-ext} and definition of
the basic extension we have
$$
J\vNN'J=\pi(\vNN_{U\otimes\bar U}).
$$
Therefore the map $x\mapsto Jx^*J$ defines a $*$-anti-isomorphism
$$
\vNN'\cap\vNN_U\cong \vNN_U'\cap\vNN_{U\otimes\bar U}.
$$
Identifying these relative commutants with $\PP(U)$ and $\PP(\bar U)$,
respectively, we get a $*$-anti-isomorphism $\PP(U)\cong\PP(\bar U)$,
which we denote by $\eta\mapsto\eta^\vee$.

\begin{lemma} \label{lvee} For every $\eta\in\PP(U)$ we have
$$
\eta^\vee=(R^*_U\otimes\iota_{\bar U})(\iota_{\bar
U}\otimes\eta\otimes\iota_{\bar U})(\iota_{\bar U}\otimes\bar R_U).
$$
\end{lemma}

\bp Consider the element $\tilde\eta=(R^*_U\otimes\iota_{\bar
U})(\iota_{\bar U}\otimes\eta\otimes\iota_{\bar U})(\iota_{\bar
U}\otimes\bar R_U)$. In terms of families of morphisms this means that
$$
\tilde\eta_X=(\iota_X\otimes R^*_U\otimes\iota_{\bar
U})(\eta_{X\otimes\bar U}\otimes\iota_{\bar
U})(\iota_X\otimes\iota_{\bar U}\otimes\bar R_U),
$$
or equivalently,
\begin{equation}\label{evee} (\iota_X\otimes
R^*_U)(\tilde\eta_X\otimes\iota_U)=(\iota_X\otimes
R^*_U)\eta_{X\otimes\bar U}.
\end{equation}

For every $n$ consider the projection $p_n\colon
L^2(\vNN_U,\omega_U)\to L^2(\vNN^{(n)}_U,\omega^{(n)}_U)$.  Let
$x\in\vNN'\cap\vNN_U$ be the element corresponding to $\eta$, and
$\tilde x\in\vNN_U'\cap\vNN_{U\otimes\bar U}$ be the element
corresponding to $\tilde\eta$. By Lemma~\ref{lpi} and the way we
represent $\tilde\eta$ by $\tilde x$, for every $y\in\vNN^{(n)}_U$ we
have
\begin{equation} \label{evee2} p_n\pi(\tilde
x)\Lambda_{\omega_U}(y)=\Lambda_{\omega_U}((\iota\otimes
R_U^*)(\tilde\eta_{V^{\otimes n}\otimes
U}\otimes\iota_U)(y\otimes\iota_{\bar U}\otimes \iota_U)(\iota\otimes
\bar R_U\otimes\iota_U)).
\end{equation} On the other hand, since $x$ is contained in the
centralizer of $\omega_U$, we have
\begin{align*}
p_nJx^*J\Lambda_{\omega_U}(y)&=p_n\Lambda_{\omega_U}(yx)=\Lambda_{\omega_U}(y
\eta_{V^{\otimes n}})\\ &=\Lambda_{\omega_U}((\iota\otimes
R_U^*)(y\otimes\iota_{\bar U}\otimes \iota_U)(\iota\otimes \bar
R_U\otimes\iota_U)\eta_{V^{\otimes n}})\\
&=\Lambda_{\omega_U}((\iota\otimes R_U^*)\eta_{V^{\otimes n}\otimes
U\otimes\bar U}(y\otimes\iota_{\bar U}\otimes \iota_U)(\iota\otimes
\bar R_U\otimes\iota_U)).
\end{align*} By \eqref{evee} the last expression equals \eqref{evee2},
so
$$
p_n\pi(\tilde x)\Lambda_{\omega_U}(y)=p_nJx^*J\Lambda_{\omega_U}(y).
$$
Since this is true for all $n$ and $y\in\vNN^{(n)}_U$, we conclude
that $\pi(\tilde x)=Jx^*J$.  \ep

\bp[Proof of Proposition~\ref{pminindex}] The operator valued weights
from $\vNN_U$ to $\vNN$ are parametrized by the positive elements $a
\in \vNN'\cap\vNN_U$ by $a \mapsto E^a$, where $E^a$ is defined by
$E^a(x) = E_U(a^{1/2} x a^{1/2})$. The map $E^a$ is a conditional
expectation if and only if the normalization condition $E_U(a)=1$
holds. Moreover, by the proof of~\cite{MR976765}*{Theorem~1},
$(E^a)^{-1}$ is given by $x \mapsto E_U^{-1}(a^{-1/2} x
a^{-1/2})$. Therefore we have
$$
[\vNN_U\colon\vNN]_0=\min_{\substack{a\in \vNN'\cap \vNN_U,\\
a>0}}E_U(a)E^{-1}_U(a^{-1})=\min_{\substack{a\in \vNN'\cap \vNN_U,\\
a>0}}d^\CC(U)^2E_U(a)\tilde E_U(Ja^{-1}J),
$$
where $\tilde E_U=d^\CC(U)^{-2}JE_U^{-1}(J\cdot
J)J\colon\pi(\vNN_{U\otimes\bar U})\to\vNN_U$.

If $a\in \vNN'\cap \vNN_U$ corresponds to $\eta\in\PP(U)$, we have
$$
E_U(a)=d^{\CC}(U)^{-1}\bar R_U(\eta\otimes\iota)\bar R_U^*
$$
By Lemma~\ref{l:basic-ext} we have $\tilde E_U(\pi(x))=\pi(E_{\bar
U}(x))$ for $x\in\vNN_{U\otimes\bar U}$. Hence by Lemma~\ref{lvee} we
get
\begin{align*} \tilde
E_U(Ja^{-1}J)&=d^{\CC}(U)^{-1}R_U^*((\eta^{-1})^\vee\otimes\iota)R_U\\
&=d^{\CC}(U)^{-1}R_U^*(R^*_U\otimes\iota_{\bar
U}\otimes\iota_U)(\iota_{\bar U}\otimes\eta^{-1}\otimes\iota_{\bar
U}\otimes\iota_U)(\iota_{\bar U}\otimes\bar R_U\otimes\iota_U)R_U\\
&=d^{\CC}(U)^{-1}R_U^*(\iota_{\bar U}\otimes\eta^{-1})R_U.
\end{align*} We thus conclude that $[\vNN_U\colon\vNN]_0$ is the
minimum of the products of the scalars
$$
\bar R_U(\eta\otimes\iota)\bar R_U^*\ \ \text{and}\ \
R_U^*(\iota\otimes\eta^{-1})R_U
$$
over all positive invertible $\eta\in\PP(U)$. This is exactly
$d^\PP(U)^2$.  \ep

\bp[Proof of Theorem~\ref{tminimal}] The estimate $\| \Gamma_U \| \le
d^\PP(U)$ comes for free.  We thus need to prove the opposite
inequality.

Let $E^\PP_U\colon\vNN_U\to\vNN$ be the minimal conditional
expectation.  Let us first assume that $\vNN_U$ (and hence $\vNN$) is
infinite.  Then by Proposition~\ref{pminindex} and
\cite{MR1096438}*{Corollary~7.2} we have the equalities
$$
2\log d^\PP(U)=\log \Ind E^\PP_U =H_{E^\PP_U}(\vNN_U | \vNN).
$$
Let $\epsilon > 0$ and $\psi$ be a normal state on $\vNN_U$ such that
$$H_{\psi}(\vNN_U | \vNN) \ge 2 \log d^\PP(U)
- \epsilon.$$

When $A$ is a finite subset of $\supp \mu$, consider the projection
$p_A = \bigoplus_{s \in A} \iota_s$ in $\vNN^{(1)}$. If $A_1, \ldots,
A_n$ are finite subsets of $\supp\mu$, then $p_{A_*} = p_{A_n} \otimes
\cdots \otimes p_{A_1}$ is a projection in $\vNN^{(n)}$, and we
consider the corresponding corner
$$
\vNN_U^{A_*} = p_{A_*}\vNN_U^{(n)} p_{A_*} = \bigoplus_{\substack{s_i,
s'_i \in A_i\\i=1,\ldots,n}} \CC(U_{s_n} \otimes \cdots \otimes
U_{s_1} \otimes U, U_{s'_n} \otimes \cdots \otimes U_{s'_1} \otimes U)
$$
in $\vNN^{(n)}_U$ and the similarly defined corner $\vNN^{A_*}$ in
$\vNN^{(n)}$. When $\psi(p_{A_*})\ne0$, define also a state
$\psi_{A_*}$ on~$\vNN_U^{A_*} $ by $
\psi_{A_*}=\psi(p_{A_*})^{-1}\psi(p_{A_*}\cdot p_{A_*}).  $ By
the lower semicontinuity of relative entropy, we can find~$n$ and finite
sets $A_1, \ldots, A_n$ such that
$$
H_{\psi_{A_*}}(\vNN^{A_*}_U | \vNN^{A_*}) \ge H_\psi(\vNN_U | \vNN) -
\epsilon.
$$
By Proposition~\ref{prop:stt-rel-entropy-graph-norm}, the inclusion
matrix $\Gamma_{A_*, U}$ of $\vNN^{A_*} \subset \vNN_U^{A_*}$
satisfies $$2 \log \| \Gamma_{A_*,U} \|\ge H_{\psi_{A_*}}(
\vNN_U^{A_*} | \vNN^{A_*}).$$ Therefore we have the estimate
$$
\log \| \Gamma_{A_*, U} \|\ge\log d^\PP(U)- \epsilon.
$$
But the transpose of the matrix $\Gamma_{A_*,U}$ is obtained from
$\Gamma_U$ by considering only columns that correspond to the simple
objects appearing in the decomposition of $U_{s_n} \otimes \cdots
\otimes U_{s_1} $ for $s_i\in A_i$, and then removing the zero
rows. Hence
$$
\|\Gamma_U \|\ge \| \Gamma_{A_*,U} \|.
$$
Since $\epsilon$ was arbitrary, we thus get $\| \Gamma_U \| \ge
d^\PP(U)$.

\smallskip

If $\vNN_U$ is finite, we consider the inclusion $\vNN \vNotimes M
\subset \vNN_U \vNotimes M$ for some infinite hyperfinite von Neumann
algebra $M$ with a prescribed strongly operator dense increasing
sequence $M_{n_k}(\C)\subset M$; for example, we could take a Powers
factor $R_\lambda$ with the usual copies of $M_2(\C)^{\otimes k}$ in
it.  Then the minimal conditional expectation $\vNN_U \vNotimes M \to
\vNN \vNotimes M$ is given by $E^\PP_U \otimes \iota$, and its index
equals that of $E^\PP_U$. Since the inclusion matrix of $\vNN^{A_*}
\otimes M_{n_k}(\C) \subset \vNN_U^{A_*} \otimes M_{n_k}(\C)$ is the
same as that of $\vNN^{A_*} \subset \vNN_U^{A_*}$, we can then argue
in the same way as above.  \ep

Since amenability of dimension functions is preserved under
homomorphisms of fusion algebras by
\cite{MR1644299}*{Proposition~7.4}, we get the following corollary.

\begin{corollary}\label{cor:from-P-eq-C-to-amen} Let
$\Pi\colon\CC\to\PP$ be the Poisson boundary of a rigid C$^*$-tensor
category with respect to an ergodic probability measure on
$\Irr(\CC)$. Then $\PP$ is an amenable C$^*$-tensor category.
\end{corollary}

Combining this with Corollary~\ref{camenb} we get the following
categorical version of the
Furstenberg--Kaimanovich--Vershik--Rosenblatt characterization of
amenability.

\begin{theorem} \label{tFKVR} A rigid C$^*$-tensor category $\CC$ is
amenable if and only if there is a probability measure~$\mu$ on
$\Irr(\CC)$ such that the Poisson boundary of $(\CC,\mu)$ is
trivial. Furthermore, the Poisson boundary of an amenable C$^*$-tensor
category is trivial for any ergodic probability measure.
\end{theorem}

Therefore we can say that while weak amenability can be detected by
studying classical Poisson boundaries of random walks on the fusion
algebra, for amenability we have to consider noncommutative, or
categorical, random walks. We can also say that nontriviality of the
Poisson boundary $\Pi\colon\CC\to\PP$ with respect to an ergodic
measure shows how far a weakly amenable category $\CC$ is from being
amenable.

\bigskip

\section{Amenable functors}

In this section we will give another characterization of amenability
in terms of invariant means. We know that on the level of fusion
algebras existence of invariant means is not enough for
amenability. Therefore we need a more refined categorical notion.

\begin{definition} Let $\CC$ be a C$^*$-tensor category and
  $F\colon\CC\to\A$ be a unitary tensor functor into a C$^*$-tensor
  category $\A$ with possibly nonsimple unit. A \emph{right invariant
    mean} for $F$ is a collection $m=(m_{U,V})_{U,V}$ of linear
  maps $$m_{U,V}\colon \hat{\CC}(U, V) \to \A(F(U),F(V))$$ that are
  natural in $U$ and $V$ and satisfy the following properties:
\begin{enumerate}
\item the maps $m_U=m_{U,U}\colon \hat{\CC}(U)\to
\End{\A}{F(U)}$ are unital and positive;

\medskip

\item for any $\eta\in \hat{\CC}(U, V)$ and any object $Y$ in
  $\CC$ we have
$$
m_{U\otimes Y,V\otimes
Y}(\eta\otimes\iota_Y)=F_2(m_{U,V}(\eta)\otimes\iota_{F(Y)});
$$

\item for any $\eta\in \hat{\CC}(U, V)$ and any object $Y$ in
  $\CC$ we have
$$
m_{Y\otimes U,Y\otimes V}(\iota_Y\otimes\eta)=F_2(\iota_{F(Y)}\otimes
m_{U,V}(\eta)).
$$
\end{enumerate}

If a right invariant mean for $F$ exists, we say that $F$ is
\emph{amenable}.
\end{definition}

Note that naturality of $m_{U,V}$ and property (i) in the above
definition easily imply that the maps~$m_U$ are completely positive,
and $m_{U,V}(\eta)^*=m_{V,U}(\eta^*)$.  As usual, we omit subscripts
and simply write $m$ instead of $m_{U, V}$ when there is no confusion.

\smallskip

The relevance of this notion for categorical random walks is explained
by the following simple observation, similar to the easy part of
Proposition~\ref{pweakamen}.

\begin{proposition} \label{ppoissonamenability} Let $\CC$ be a rigid
C$^*$-tensor category, $\mu$ be a probability measure on $\Irr(\CC)$,
and $\Pi\colon\CC\to\PP$ be the Poisson boundary of~$(\CC,\mu)$. Then
the functor $\Pi\colon\CC\to\PP$ is amenable.
\end{proposition}

\bp Fix a free ultrafilter $\omega$ on $\N$, and then define
$$
m(\eta)_X=\lim_{n\to\omega}\frac{1}{n}\sum^{n-1}_{k=0}P^k_\mu(\eta)_X.
$$
All the required properties of a right invariant mean follow
immediately by definition. For example, property (iii) in the
definition follows from the identity $P_X(\iota_Y\otimes
\eta)=\iota_Y\otimes P_X(\eta)$.  \ep

For functors into categories with nonsimple units we do not have much
insight into the meaning of amenability. But if we fall back to our
standard assumption of simplicity of tensor units, we have the
following result.

\begin{theorem}\label{tamnefunctor} Let $\A$ and $\CC$ be rigid C$^*$-tensor
categories with simple units and $F\colon\CC\to\A$ be a unitary tensor functor. Then~$F$
is amenable if and only if $\CC$ is weakly amenable and $d^\A F$ is
the amenable dimension function on~$\CC$.
\end{theorem}

Let $F\colon\CC\to\A$ be an amenable unitary tensor functor with a
right invariant mean $m$. For simplicity we assume as usual that $\CC$
and $\A$ are strict and $F$ is an embedding functor. Let us start by
showing that existence of $F$ implies weak amenability.

\begin{lemma} \label{lweakamen} The linear functional
$m_\un\colon\hat\CC(\un)\cong\ell^\infty(\Irr(\CC))\to\End{\A}{\un}\cong\C$
is a right invariant mean on the fusion algebra of $\CC$ equipped with
the dimension function $d^\CC$.
\end{lemma}

\bp In addition to the operators $P_X$ on $\hat\CC(\un)$ we
normally use, we also have the operators $Q_X$ given by
$$
Q_X(\eta)_Y=d^\CC(X)^{-1}(\iota_Y\otimes\tr_X)(\eta_{Y\otimes
X})=d^\CC(X)^{-1} (\iota_Y\otimes\bar R_X^*)(\eta_{Y\otimes
X}\otimes\iota_{\bar X})(\iota_Y\otimes \bar R_X),
$$
where $(R_X,\bar R_X)$ is a standard solution of the conjugate
equations for $X$ in $\CC$. Since
$$
\eta_{Y\otimes X}\otimes\iota_{\bar
X}=(\iota_X\otimes\eta\otimes\iota_{\bar X})_Y,
$$
we can write this as
$$
Q_X(\eta)=d^\CC(X)^{-1}\bar R^*_X(\iota_X\otimes\eta\otimes\iota_{\bar
X})\bar R_X.
$$
Applying the invariant mean we get
$$
m(Q_X(\eta))=d^\CC(X)^{-1}\bar R^*_X(\iota_X\otimes
m(\eta)\otimes\iota_{\bar X})\bar R_X=m(\eta).
$$
Thus $m_\un$ is a right invariant mean on
$\ell^\infty(\Irr(\CC))$.  \ep

Since $\CC$ is weakly amenable, we can choose an ergodic probability
measure and consider the corresponding Poisson boundary
$\Pi\colon\CC\to\PP$. We then have the following result, which has its
origin in Tomatsu's considerations in~\cite{MR2335776}*{Section~4}.

\begin{lemma} \label{lmult2} For every object $U$ in $\CC$ the map
$\Lambda_U\colon\End{\PP}{U}\to\End{\A}{U}$ obtained by
restricting~$m_U$ to~$\End{\PP}{U}$ is multiplicative.
\end{lemma}

\bp Recall that in Section~\ref{suniversal} we constructed faithful
unital completely positive maps
$\Theta_U\colon\End{\A}{U}\to\End{\PP}{U}$, $\Theta_U(T)_X=E_{X\otimes
U}(\iota\otimes T)$. By faithfulness of $\Theta_U$, the multiplicative
domain of $\Theta_U\Lambda_U$ is contained in that of
$\Lambda_U$. Therefore in order to prove the lemma it suffices to show
that $\Theta_U\Lambda_U$ is the identity map.

Let us show first that for any $\eta\in\End{\PP}{U}$ we have $\Theta
\Lambda(\eta)_\un=\eta_\un$, that is,
$$
E(m(\eta))=\eta_\un.
$$
Take $S\in\End{\CC}{U}$. Then we have
$$
\tr_U(SE(m(\eta)))=\psi_U(m(S\eta))=d^\CC(U)^{-1}\bar
R_U^*(m(S\eta)\otimes\iota)\bar R_U =d^\CC(U)^{-1}m(\bar
R_U^*(S\eta\otimes\iota)\bar R_U).
$$
Since the element $\bar R_U^*(S\eta\otimes\iota)\bar R_U$ lies in
$\End{\PP}{\un}$, it is scalar. This scalar must be equal to
$$\bar R_U^*(S\eta_\un\otimes\iota)\bar R_U=\tr_U(S\eta_\un).$$
Hence we obtain
$$
\tr_U(SE(m(\eta)))=\tr_U(S\eta_\un),
$$
and since this is true for all $S$, we get $\Theta
\Lambda(\eta)_\un=\eta_\un$.

Now, for any object $X$ in $\CC$, we use the above equality for
$\iota_X\otimes\eta$ instead of $\eta$ and get
$$
\Theta \Lambda(\eta)_X=E(\iota_X\otimes
m(\eta))=E(m(\iota_X\otimes\eta))=\Theta
\Lambda(\iota_X\otimes\eta)_\un =(\iota_X\otimes\eta)_\un=\eta_X,
$$
which implies the desired equality $\Theta \Lambda(\eta)=\eta$.  \ep

\bp[Proof of Theorem~\ref{tamnefunctor}] Consider an amenable unitary
tensor functor $F\colon\CC\to\A$.
By Lemma~\ref{lweakamen} we know that
$\CC$ is weakly amenable. By Lemma~\ref{lmult2} and the definition of
invariant means, any right invariant mean for $F$ defines a strict
unitary tensor functor $\Lambda\colon\tilde\PP\to\A$, where
$\tilde\PP\subset\PP$ is the full subcategory consisting of objects in
$\CC$. Extend this functor to a unitary tensor functor
$\Lambda\colon\PP\to\A$. Then $d^\A(U)\le d^\PP(U)$ for any object $U$
in $\CC$, but since by Theorem~\ref{tminimal} the dimension function
$d^\PP\Pi$ on $\CC$ is amenable, we conclude that
$d^\A(U)=d^\PP(U)=\|\Gamma_U\|$.

\smallskip

Conversely, assume $\CC$ is weakly amenable and $F\colon\CC\to\A$ is a
unitary tensor functor such that $d^\A F$ is the amenable dimension
function. Then by Theorem~\ref{tuniversal1} there exists a unitary
tensor functor $\Lambda\colon\PP\to\A$ such that $\Lambda\Pi\cong
F$. By Proposition~\ref{ppoissonamenability} there exists a right
invariant mean for the functor $\Pi\colon\CC\to\PP$. Composing it with
the functor $\Lambda$ we get a right invariant mean for $\Lambda\Pi$,
from which we get a right invariant mean for $F$.  \ep

Applying Theorem~\ref{tamnefunctor} to the identity functor we get a
characterization of amenability of tensor categories in terms of
invariant means.

\begin{theorem} A rigid C$^*$-tensor category $\CC$ is amenable if and
only if the identity functor $\CC\to\CC$ is amenable.
\end{theorem}

Note that by the proof of Theorem~\ref{tamnefunctor}, given an
amenable C$^*$-tensor category $\CC$, we can construct a right
invariant mean for the identity functor as follows. Choose an ergodic
probability measure $\mu$ on $\Irr(\CC)$ and a free ultrafilter
$\omega$ on $\N$. Then we can define
$$
m(\eta)=\lim_{n\to\omega}\frac{1}{n}\sum^{n-1}_{k=0}
P^k_\mu(\eta)_\un.
$$
On the other hand, the construction of a right invariant mean for a
functor $F\colon\CC\to\A$ such that $\CC$ is weakly amenable, but not
amenable, and $d^\A F$ is the amenable dimension function, is more
elusive, as it relies on the existence of a factorization of $F$
through the Poisson boundary $\CC\to\PP$.

\bigskip

\section{Amenability of quantum groups and subfactors}

In this section we apply some of our results to categories considered
in the theory of compact quantum groups and in subfactor theory.

\subsection{Quantum groups}

Let $G$ be a compact quantum group. We follow the conventions
of~\cite{MR3204665}. In particular, the algebra $\C[G]$ of
regular functions on $G$ is a Hopf $*$-algebra, and by a finite
dimensional unitary representation of $G$ we mean a unitary element
$U\in B(H_U)\otimes\C[G]$, where $H_U$ is a finite dimensional Hilbert
space, such that $(\iota\otimes\Delta)(U)=U_{12}U_{13}$. Finite
dimensional unitary representations form a rigid C$^*$-tensor category
$\Rep G$, with the tensor product of $U$ and $V$ defined by
$U_{13}V_{23}\in B(H_U)\otimes B(H_V)\otimes\C[G]$. The categorical
dimension of $U$ is equal to the quantum dimension, given by the trace  $\Tr(\rho_U)$ of the Woronowicz character.

\smallskip

Recall that $G$ is called \emph{coamenable} if $\|\Gamma_U\|=\dim H_U$
for every finite dimensional unitary representation $U$. There are a
number of equivalent conditions, but using this definition as our
starting point we immediately get that
$$
\Rep G \ \ \text{is amenable}\Leftrightarrow G\ \ \text{is coamenable,
  and of Kac type}.
$$

Coamenability of $G$ is known to be equivalent to amenability of the
dual discrete quantum group~$\hat G$. Recall that the algebra of
bounded functions on $\hat G$ is defined by $\ell^\infty(\hat
G)=\ell^\infty\text{-}\bigoplus_{s\in\Irr(G)}B(H_s)$, and the coproduct
$\Dhat\colon\ell^\infty(\hat G)\to\ell^\infty(\hat
G)\vNotimes\ell^\infty(\hat G)$ is defined by duality from the product
on $\C[G]$, if we view $\ell^\infty(\hat G)$ as a subspace of
$\C[G]^*$ by associating to a functional $\omega\in\C[G]^*$ the
collection of operators $\pi_s(\omega)=(\iota\otimes\omega)(U_s)\in
B(H_s)$, $s\in\Irr(G)$. The quantum group $\hat G$ is called \emph{amenable},
if there exists a right invariant mean on $\hat G$, that is, a state
$m$ on $\ell^\infty(\hat G)$ such that
$$
m(\iota\otimes\phi)\Dhat=\phi(\cdot)1\ \ \text{for any normal linear
functional}\ \phi\ \text{on}\ \ell^\infty(\hat G).
$$
The restriction of such an invariant mean to $Z(\ell^\infty(\hat
G))\cong\ell^\infty(\Irr(G))$ defines a right invariant mean on the
fusion algebra of $\Rep G$ equipped with the quantum dimension
function. Therefore
$$
G\ \ \text{is coamenable}\Leftrightarrow\hat G \ \ \text{is amenable}\Rightarrow \Rep G\ \ \text{is weakly
amenable}.
$$

Among various known characterizations of coamenability the implication
($\hat G$ is amenable $\Rightarrow$ $G$ is coamenable) is probably the
most nontrivial. This was proved independently
in~\cite{MR2276175}*{Theorem~3.8}
and in~\cite{MR2113848}*{Corollary~9.6}. We will show now that our
results on amenable functors are generalizations of this.

\begin{theorem} If $\hat G$ is amenable, then the forgetful functor
$F\colon\Rep G\to\Hilb_f$ is amenable, and therefore $G$ is
coamenable.
\end{theorem}

\bp We will only consider the case when $\Irr(G)$ is at most
countable, so that $\Rep G$ satisfies our standing assumptions, the
general case can be easily deduced from this.

As discussed in \cite{MR3291643}*{Section~4.1}, the space
$\hat{\CC}(U, V)$ can be identified with the space of elements
$$
\eta\in \ell^\infty(\hat G)\otimes B(H_U,H_V)\ \ \text{such that}\ \
V^*_{31}(\alpha\otimes\iota)(\eta)U_{31}=1\otimes \eta,
$$
where $\alpha\colon \ell^\infty(\hat G)\to L^\infty(G)\vNotimes
\ell^\infty(\hat G)$ is the left adjoint action of $G$. Under this
identification we have
$$
\iota_Y\otimes\eta=(\iota\otimes\pi_Y\otimes\iota)(\Dhat\otimes\iota)(\eta),
$$
where $\pi_Y\colon\ell^\infty(\hat G)\to B(H_Y)$ is the representation
defined by $Y$, while the element $\eta\otimes\iota_Y$ has the obvious
meaning. From this we immediately see that if $m$ is a right invariant
mean on $\hat G$, then the maps $m\otimes\iota\colon \ell^\infty(\hat
G)\otimes B(H_U,H_V)\to B(H_U,H_V)$ define a right invariant mean for
$F$. Thus $F$ is amenable. By Theorem~\ref{tamnefunctor} we conclude
that $\|\Gamma_U\|=\dim F(U)=\dim H_U$ for every $U$, so $G$ is
coamenable.  \ep

As for the Poisson boundary of $\Rep G$, from the universal property of the Poisson boundary it is easy to
deduce that if $G$ is coamenable (and so $\Rep G$ is weakly amenable), then the Poisson boundary of $\Rep G$ with respect to any
ergodic measure is the forgetful functor $\Rep G\to \Rep K$, where
$K\subset G$ is the maximal quantum subgroup of $G$ of Kac type. This
will be discussed in detail in~\cite{MR3556413}.

\subsection{Subfactor theory}

Let $N\subset M$ be a finite index inclusion of II$_1$-factors. Denote
by $\tau$ the tracial state on $M$, and by $E$ the trace-preserving
conditional expectation $M\to N$. We denote $[M : N]=\Ind E$, and
the minimal index of $N \subset M$ by $[M : N]_0$. Put $M_{-1}=N$,
$M_0=M$, and choose a tunnel
$$
\dots\subset M_{-3}\subset M_{-2}\subset M_{-1}\subset M_0,
$$
so that $M_{-n+1}$ is the basic extension of $M_{-n-1}\subset M_{-n}$
for all $n\ge1$. For every $j\le 1$ denote by $M_j^\st\subset M_j$ the
$s^*$-closure of $\bigcup_{n\ge1}(M_{j-n}'\cap M_j)$ with respect to the
restriction of $\tau$. The inclusion $N^\st\subset M^\st$ of finite
von Neumann algebras is called a standard model of $N\subset
M$~\cite{MR1278111}.

Let $\BB_N(M)$ be the full C$^*$-tensor subcategory of the category
$\Hilb_N$ of Hilbert bimodules over~$N$ generated by~$L^2(M)$.
Let $M_1$ be the basic extension of $N\subset M$, so that
$\Endd_\NN(L^2(M))\cong N'\cap M_1$.  The embedding $N \to M_1$
induces a morphism $L^2(N) \to L^2(M) \otimes_N L^2(M)$ in $\BB_N(M)$,
which defines a solution of the conjugate equations for $L^2(M)$ up to a scalar normalization. Moreover, it can be shown (compare with Proposition~\ref{pminindex}) that the categorical trace corresponds to the minimal conditional expectation $M_1\to N$, and consequently $d(L^2(M))=[M_1 : N]^{1/2}_0=[M : N]_0$.  It is also known, see Proposition~\ref{pocneanu}, that the inductive system of the algebras
$\Endd_\NN(L^2(M)^{\otimes_Nn})$, with respect to the embeddings
$T\mapsto \iota_{L^2(M)}\otimes T$, can be identified with
$(M_{-2n+1}'\cap M_1)_{n\ge1}$ in such a way that the shift
endomorphism $T\mapsto T\otimes\iota_{L^2(M)}$ of
$\bigcup_{n\ge1}\Endd_\NN(L^2(M)^{\otimes_Nn})$ corresponds to the
endomorphism $\gamma^{-1}$ of $\bigcup_{n\ge1}(M_{-2n+1}'\cap M_1)$,
where $\gamma$ is the canonical shift.

The normalized categorical trace on $\Endd_\NN(L^2(M))$ defines a
probability measure~$\mu_\st$ on the set of isomorphism classes of
simple submodules of $L^2(M)$. More explicitly, it can be shown that
the value of the normalized categorical trace on any minimal
projection $p\in N'\cap M_1$ equals
$$
(\tau(p)\tau'(p))^{1/2}\frac{[M : N]}{[M : N]_0},
$$
where $\tau'$ is the unique tracial state on $N'\subset B(L^2(M))$. See~\cite{MR976765}*{Section~2}
and~\cite{MR1278111}*{Section~1.3.6} for related results. Then the
measure~$\mu_\st$ is defined~by
$$
\mu_\st([pL^2(M)])=m_p(\tau(p)\tau'(p))^{1/2}\frac{[M :
N]}{[M : N]_0},
$$
where $m_p$ is the multiplicity of $pL^2(M)$ in $L^2(M)$.

Recall that an inclusion for which $[M : N]=[M : N]_0$, is called
extremal.  From the above considerations, unless $N\subset M$ is extremal,
we see that the categorical trace defines a tracial state of
$\bigcup_{n\ge1}(M_{-2n+1}'\cap M_1)$ that is different from~$\tau$.

\smallskip

Let us first review what our results say about $(\BB_N(M),\mu_\st)$ for extremal inclusions. From the identification of $\bigcup_{n\ge1}\Endd_\NN(L^2(M)^{\otimes_Nn})$  with $\bigcup_{n\ge1}(M_{-2n+1}'\cap M_1)$ we conclude that the von Neumann algebra $\vNN_{L^2(M)}$ constructed in Section~\ref{sec:longo-roberts-appr} is
isomorphic to $M_1^\st$. More precisely, we take $V=L^2(M)$ for the
construction of~$\vNN_{L^2(M)}$, so unless $N'\cap M_1$ is abelian, we
apply the modification of our construction of the algebras~$\vNN_U$
discussed in Remark~\ref{rmultiplicity}. The subalgebra
$\vNN\subset\vNN_{L^2(M)}$ corresponds then to $N^\st=\gamma^{-1}(M_1^\st)\subset
M_1^\st$. In particular, $N^\st$ is a factor if and only if $\mu_\st$ is ergodic. Proposition~\ref{prop:p-bdry-as-rel-comm-LR0}
translates now into the following statement, which is closely related to a
result of Izumi~\cite{MR2059809}.

\begin{proposition} \label{pIzumi} Let $\Pi\colon\BB_N(M)\to\PP$ be the Poisson boundary of   $(\BB_N(M),\mu_\st)$. Then, assuming that $N\subset M$ is extremal, we have
$$
\PP(L^2(M))\cong ({N^\st})'\cap M_1^\st.
$$
\end{proposition}

More generally, by the same argument we have
$\PP(L^2(M)^{\otimes_Nn})\cong ({M_{-2n+1}^\st})'\cap M_1^\st$. Since
$L^2(M)$ contains a copy of the unit object induced by the inclusion
$N \to M$, we have $\mu_\st(e)>0$. Hence the supports of $\mu^n_\st$
are increasing, and therefore the isomorphisms
$\PP(L^2(M)^{\otimes_Nn})\cong ({M_{-2n+1}^\st})'\cap M_1^\st$
completely describe the morphisms in the category $\PP$. In fact,
recalling that $M_1^\st$ is the basic extension of $N^\st\subset
M^\st$, see ~\cite{MR1278111}*{Section~1.4.3}, we may conclude that
$\PP$ can be identified with $\BB_{N^\st}(M^\st)$. We leave it to the
interested reader to find a good description of the functor $\Pi\colon
\BB_N(M)\to\BB_{N^\st}(M^\st)$.

\smallskip

Consider the principal graph $\Gamma_{N,M}$ of $N\subset M$. Then
$\Gamma_{L^2(M)}$ can be identified with $\Gamma_{M,M_1}
\Gamma_{M,M_1}^t$. Recall also that we have the equality
$\|\Gamma_{N,M}\| = \|\Gamma_{M,M_1}\|$
by~\cite{MR1278111}*{Section~1.3.5}.

Turning now to Theorem~\ref{tminimal} and Proposition~\ref{pminindex},
we get the following result (again, to be more precise we use the
modification of the construction of~$\vNN_U$ described in
Remark~\ref{rmultiplicity}).

\begin{theorem} \label{tPopaext} Assume $N \subset M$ is extremal and
  $N^\st$ is a factor. Then we have
$$
\|\Gamma_{N,M}\|^4=[M_1^\st : N^\st]_0.
$$
\end{theorem}

If $M^\st$ is also a factor, this can of course be formulated as
$\|\Gamma_{N,M}\|^2=[M^\st : N^\st]_0$.

\smallskip

Applying Theorem~\ref{tFKVR} we get the following result, which recovers part of Popa's characterization of extremal subfactors with strongly amenable standard invariant~\cite{MR1278111}*{Theorem~5.3.1}.

\begin{theorem}\label{thm:amen-equiv-high-rel-comm-std-mdl}
  Assume $N \subset M$ is extremal. The following conditions are
  equivalent:
\begin{enumerate}
\item $N^\st$ is a factor and $\|\Gamma_{N,M}\|^2=[M : N]_0$;
\item $(M^\st_{-2n+1})'\cap M_1^\st=M_{-2n+1}'\cap M_1$ for all
$n\ge1$.
\end{enumerate}
\end{theorem}

\bp As we already observed, the condition that $N^\st$ is a factor in (i) means exactly that the measure $\mu_\st$ is ergodic. The condition $\|\Gamma_{N,M}\|^2=[M : N]_0$ means that $\|\Gamma_{L^2(M)}\|=d(L^2(M))$. Since the module $L^2(M)$ is self-dual and generates $\BB_N(M)$, this condition is equivalent to amenability of $\BB_N(M)$.

On the other hand, by Proposition~\ref{pIzumi} and its extension to
the modules $L^2(M)^{\otimes_Nn}$ discussed above, condition (ii) is
equivalent to triviality of the Poisson boundary of
$(\BB_N(M),\mu_\st)$.

This shows that the equivalence of (i) and (ii) is indeed a
consequence of Theorem~\ref{tFKVR}.  \ep

If we write the proof of the implication (ii)$\Rightarrow$(i) in terms
of the algebras $M_{-2n+1}'\cap M_1$ instead of
$\Endd_\NN(L^2(M)^{\otimes_Nn})$, we get an argument similar to Popa's
proof based on~\cite{MR1111570}, which was our inspiration. On the
other hand, our proof of (i)$\Rightarrow$(ii) seems to be very
different from his arguments.

\smallskip

Next, let us comment on the nonextremal case. One possibility is to consider the completion of $\bigcup_{n\ge1}(M_{j-n}'\cap M_j)$ with respect to the trace induced by the minimal conditional expectation (that is, the categorical trace) instead of $\tau$. Then all the above statements continue to hold if we replace $N^\st$ and $M^\st$ by the corresponding new von Neumann
algebras. Note that the inclusion $N^\st\subset M^\st$ defined this way is the standard model in the
conventions of~\cite{MR1339767}. Then, for example, the implication (i)$\Rightarrow$(ii) in Theorem~\ref{thm:amen-equiv-high-rel-comm-std-mdl}
corresponds to~\cite{MR1339767}*{Lemma~5.2}.

But some results, notably~Theorem~\ref{tPopaext}, continue to hold for the
inclusion $N^\st \subset M^\st$ defined with respect to $\tau$ in the
nonextremal case also. The proof goes in basically the same way as in the extremal case, by noting that the proof of the inequality $\|\Gamma_U\|^2\ge [\vNN_U :\vNN]_0$ in Theorem~\ref{tminimal} did not depend on how exactly the inductive limit of the algebras $\CC(V^{\otimes n}\otimes U)$ was completed to
get the factors $\vNN\subset\vNN_U$. Therefore we have
$$
\|\Gamma_{N,M}\|^4\ge [M_1^\st : N^\st]_0.
$$
The opposite inequality can be proved either by realizing that the
dimension function on $\BB_{N^\st}(M_1^\st)$ defines a dimension
function on $\BB_N(M_1)$, or by the following string of
(in)equalities:
$$
\|\Gamma_{N,M}\|^4=\|\Gamma_{N,M_1}\|^2\le
\|\Gamma_{N^\st,M_1^\st}\|^2 \le [M_1^\st : N^\st]_0,
$$
compare with~\cite{MR1278111}*{p.~235}. We remark that from this one
can easily obtain the implication (ii)$\Rightarrow$(vii)
in~\cite{MR1278111}*{Theorem~5.3.2} promised in~\cite{MR1278111}.

\bigskip

\appendix

\section{Estimating relative entropy}

In this appendix we estimate the relative entropy for
embeddings of finite dimensional C$^*$-algebras.

\smallskip

Let $N\subset M$ be a unital inclusion of finite dimensional
C$^*$-algebras, $\{z_k\}_{k\in K}$ be the minimal central projections
of $N$, and $\{w_l\}_{l\in L}$ be the minimal central projections of
$M$. Let $A=(a_{kl})_{k,l}$ be the multiplicity matrix of the
inclusion $N\subset M$, so that $a_{kl}=0$ if $z_kw_l=0$ and
$(Nz_kw_l)'\cap (z_kMw_lz_k)\cong\Mat_{a_{kl}}(\C)$ otherwise, and
$\{n_k\}_k$ be the dimension function of $N$, so
$Nz_k\cong\Mat_{n_k}(\C)$. The following proposition generalizes
results of Pimsner and Popa for tracial states in \cite
{MR860811}*{Section~6}.

\begin{proposition}\label{ppipo1} For any state $\varphi$ on $M$ we
have
$$
H_\varphi(M|N)\le\sum_{k,l}\varphi(z_kw_l)\log\frac{\varphi(z_k)\varphi(w_l)a_{kl}\min\{a_{kl},n_k\}}{\varphi(z_kw_l)^2},
$$
and the equality holds if $\varphi$ is tracial.
\end{proposition}

The proof follows closely the proof for tracial states given in
\cite{MR2251116}*{Theorem~10.1.4}. The key part is the following
estimate.

\begin{lemma}\label{lentropy} For any positive linear functional
$\psi$ on $M$ we have
$$
-\sum_{k,l}\psi(z_kw_l)\log\frac{\psi(w_l)\min\{a_{kl},n_k\}}{\psi(z_kw_l)}\le
S(\psi)-S(\psi|_N)\le\sum_{k,l}\psi(z_kw_l)\log\frac{\psi(z_k)a_{kl}}{\psi(z_kw_l)}.
$$
\end{lemma}

\bp Put $M_{kl}=z_kMw_lz_k$ and $N_{kl}=Nz_kw_l$. For
$\psi(z_kw_l)\ne0$ consider the state
$$
\psi_{kl}=\psi(z_kw_l)^{-1}\psi|_{M_{kl}}
$$
on $M_{kl}$.  As usual in entropy theory, it is
convenient to define a function $\eta$ by $\eta(t)=-t\log t$ for
$t\ge0$. We will constantly use the obvious equality
$$
S(\omega)=\omega(1)S(\omega(1)^{-1}\omega)+\eta(\omega(1))
$$
for positive linear functionals $\omega$. Recall that the von Neumann entropy is defined by $S(\omega)=\Tr(\eta(Q_\omega))$.

\smallskip

Let us start by estimating $S(\psi)$. We have
$$
S(\psi)=\sum_lS(\psi|_{Mw_l})=\sum_l\psi(w_l)S(\psi(w_l)^{-1}\psi|_{Mw_l})+\sum_l\eta(\psi(w_l)).
$$
By \cite{MR2251116}*{Lemma~2.2.4} applied to the projections $z_kw_l$
in $Mw_l$ we have
$$
S(\psi(w_l)^{-1}\psi|_{Mw_l})\ge\sum_k
S(\psi(w_l)^{-1}\psi|_{M_{kl}})-\sum_k\eta\left(\frac{\psi(z_kw_l)}{\psi(w_l)}\right)
=\sum_k\frac{\psi(z_kw_l)}{\psi(w_l)}S(\psi_{kl}).
$$
It follows that
\begin{equation}\label{ea1} S(\psi)\ge
\sum_{k,l}\psi(z_kw_l)S(\psi_{kl})+\sum_l\eta(\psi(w_l)).
\end{equation}

On the other hand, since $\bigoplus_kM_{kl}$ is a subalgebra of $Mw_l$ of
full rank, by \cite{MR2251116}*{Theorem~2.2.2(vii)} we have
$$
S(\psi|_{Mw_l})\le\sum_k
S(\psi|_{M_{kl}})=\sum_k\psi(z_kw_l)S(\psi_{kl})+\sum_k\eta(\psi(z_kw_l)).
$$
Therefore
\begin{equation}\label{ea2} S(\psi)\le
\sum_{k,l}\psi(z_kw_l)S(\psi_{kl})+\sum_{k,l}\eta(\psi(z_kw_l)).
\end{equation}

\smallskip

Turning to $S(\psi|_N)$, by \cite{MR2251116}*{Theorem~2.2.2(ii)} we
have
\begin{equation} \label{ea3}
S(\psi|_N)=\sum_kS(\psi|_{Nz_k})\le\sum_{k,l}S(\psi(\cdot\,
w_l)|_{Nz_k})
=\sum_{k,l}\psi(z_kw_l)S(\psi_{kl}|_{N_{kl}})+\sum_{k,l}\eta(\psi(z_kw_l)).
\end{equation}

On the other hand, since the von Neumann entropy is concave by
\cite{MR2251116}*{Theorem~2.2.2(ii)}, we have
\begin{align*}
S(\psi|_{Nz_k})&=\psi(z_k)S\left(\sum_l\frac{\psi(z_kw_l)}{\psi(z_k)}\psi(z_kw_l)^{-1}\psi(\cdot\,w_l)|_{Nz_k}\right)
+\eta(\psi(z_k))\\ &\ge
\sum_l\psi(z_kw_l)S(\psi_{kl}|_{N_{kl}})+\eta(\psi(z_k)).
\end{align*} Therefore
\begin{equation} \label{ea4}
S(\psi|_N)\ge\sum_{k,l}\psi(z_kw_l)S(\psi_{kl}|_{N_{kl}})+\sum_k\eta(\psi(z_k)).
\end{equation}

\smallskip

Now, by \cite{MR2251116}*{Theorem~2.2.2(vi)} we have
$$
|S(\psi_{kl})-S(\psi_{kl}|_{N_{kl}})|\le \log a_{kl}.
$$
This, together with \eqref{ea2} and \eqref{ea4}, gives
$$
S(\psi)-S(\psi|_N)\le \sum_{k,l}\psi(z_kw_l)\log
a_{kl}+\sum_{k,l}\eta(\psi(z_kw_l))- \sum_k\eta(\psi(z_k)),
$$
which is what we need as
$$
\eta(\psi(z_k))=-\sum_l\psi(z_kw_l)\log\psi(z_k).
$$

The lower bound for $S(\psi)-S(\psi|_N)$ follows similarly from
\eqref{ea1} and \eqref{ea3}, if we in addition use that
$$
S(\psi_{kl})-S(\psi_{kl}|_{N_{kl}})\ge -S(\psi_{kl}|_{N_{kl}})\ge-\log
n_k,
$$
so that
$$
S(\psi_{kl})-S(\psi_{kl}|_{N_{kl}})\ge -\log\min\{a_{kl},n_k\}.
$$
\ep

\bp[Proof of Proposition~\ref{ppipo1}] Given a finite decomposition
$\varphi=\sum_i\varphi_i$ we want to obtain an upper bound~on
$$
S(\varphi)-S(\varphi|_N)+\sum_i(S(\varphi_i|_N)-S(\varphi_i)).
$$
By Lemma~\ref{lentropy} we have
$$
S(\varphi)-S(\varphi|_N)\le
\sum_{k,l}\varphi(z_kw_l)\log\frac{\varphi(z_k)a_{kl}}{\varphi(z_kw_l)}
$$
and
$$
\sum_i(S(\varphi_i|_N)-S(\varphi_i))\le
\sum_{i,k,l}\varphi_i(z_kw_l)\log\frac{\varphi_i(w_l)\min\{a_{kl},n_k\}}{\varphi_i(z_kw_l)}.
$$
Since $\sum_i\varphi_i(z_kw_l)=\varphi(z_kw_l)$, using the concavity of
$\log$ we get
\begin{align*}
\sum_{i}\varphi_i(z_kw_l)\log\frac{\varphi_i(w_l)\min\{a_{kl},n_k\}}{\varphi_i(z_kw_l)}
&\le
\varphi(z_kw_l)\log\left(\sum_{i}\frac{\varphi_i(z_kw_l)}{\varphi(z_kw_l)}
\frac{\varphi_i(w_l)\min\{a_{kl},n_k\}}{\varphi_i(z_kw_l)}\right)\\
&=\varphi(z_kw_l)\log\frac{\varphi(w_l)\min\{a_{kl},n_k\}}{\varphi(z_kw_l)}.
\end{align*} Putting all this together we get the required upper bound
on $H_\varphi(M|N)$. That this bound is exactly the value of
$H_\varphi(M|N)$ for tracial $\varphi$ is proved in
\cite{MR860811}*{Section~6}, see also
\cite{MR2251116}*{Theorem~10.1.4}.  \ep

The following result generalizes another estimate of Pimsner and Popa
for tracial states, given in~\cite{MR1111570}*{Theorem~2.6}.

\begin{proposition}\label{prop:stt-rel-entropy-graph-norm} For any
state $\varphi$ on $M$ we have
$$
H_\varphi(M|N)\le2\log\|A\|.
$$
\end{proposition}

\bp Consider the sets $\Omega=\{(k,l)\in K\times L\mid a_{kl}\ne0\}$
and
$$
\Delta=\{\xi=(\xi_{kl})_{(k,l)\in\Omega}\mid \xi_{kl}\ge0,\
\sum_{k,l}\xi_{kl}=1\}\subset\R^\Omega_+.
$$
Define a function $f$ on $\R^\Omega_+$ by
$$
f(\xi)=\sum_{k,l}\xi_{kl}\log\frac{\xi^{(1)}_k\xi^{(2)}_la_{kl}^2}{\xi_{kl}^2},
$$
where $\xi^{(1)}_k=\sum_l\xi_{kl}$ and
$\xi^{(2)}_l=\sum_k\xi_{kl}$. By Proposition~\ref{ppipo1} we have
$H_\varphi(M|N)\le f(\xi)$ for $\xi\in\Delta$ defined by
$\xi_{kl}=\varphi(z_kw_l)$. Therefore it suffices to show that
$f(\xi)\le2\log\|A\|$ for all $\xi\in\Delta$. We will prove that this
is the case for any nonzero matrix $A$ with nonnegative real
coefficients.

\smallskip

Let $\zeta\in\Delta$ be a maximum point of the function
$f|_\Delta$. We may assume that $\zeta_{kl}>0$ for all
$(k,l)\in\Omega$, since otherwise we can simply modify the matrix $A$
by letting $a_{kl}=0$ for $(k,l)\in\Omega$ such that $\zeta_{kl}=0$,
which can only decrease the norm of $A$, since by the Perron--Frobenius
theory the norm of $A^*A$ is the maximum of the numbers $\mu\ge0$ such
that $A^*Aw\ge\mu w$ for some nonzero vector $w\in\R_+^L$. By removing
zero rows and columns of $A$ we may also assume that the projection
maps $\Omega\to K$ and $\Omega\to L$ are surjective, so the numbers
$\zeta^{(1)}_k$ and $\zeta^{(2)}_l$ are well-defined and strictly
positive for all $k\in K$ and $l\in L$.

Using that
$$
\frac{\partial}{\partial\xi_{kl}}\log
\xi^{(1)}_i=\frac{\delta_{ik}}{\xi^{(1)}_k}\ \ \text{and}\ \
\frac{\partial}{\partial\xi_{kl}}\log
\xi^{(2)}_j=\frac{\delta_{jl}}{\xi^{(2)}_l},
$$
we get
$$
\frac{\partial f}{\partial\xi_{kl}}(\zeta)=\log
\frac{\zeta^{(1)}_k\zeta^{(2)}_la_{kl}^2}{\zeta_{kl}^2}.
$$
Since $\zeta$ is a maximum point of $f|_\Delta$, the gradient of $f$
at this point is orthogonal to $\Delta$, so
$$
\log \frac{\zeta^{(1)}_k\zeta^{(2)}_la_{kl}^2}{\zeta_{kl}^2}=\lambda \
\ \text{for all}\ \ (k,l)\in\Omega
$$
for some $\lambda\in\R$. Then $f(\zeta)=\lambda$, and it remains to
show that $\lambda\le2\log\|A\|$.

Put
$$
v_k=(\zeta^{(1)}_k)^{1/2}\ \ \text{and}\ \ w_l=(\zeta^{(2)}_l)^{1/2}.
$$
Then, using that $\zeta_{kl}=e^{-\lambda/2}a_{kl}v_kw_l$, we get
$$
\sum_l
a_{kl}w_l=e^{\lambda/2}\sum_l\frac{\zeta_{kl}}{v_k}=e^{\lambda/2}\frac{\zeta^{(1)}_k}{v_k}
=e^{\lambda/2}v_k,
$$
so that $Aw=e^{\lambda/2}v$. Similarly we get
$A^*v=e^{\lambda/2}w$. We thus see that $w$ is an eigenvector of
$A^*A$ with eigenvalue~$e^\lambda$. Hence $e^\lambda\le\|A\|^2$.  \ep

Note that the maximum of the function $f|_\Delta$ from the above proof
is exactly $2\log\|A\|$. Indeed, let $w\in\R_+^L$ be an eigenvector of
$A^*A$ with eigenvalue $\|A\|^2$ normalized so that $\|w\|_2=1$, which
exists by the Perron--Frobenius theorem. Then letting
$\xi_{kl}=\|A\|^{-2}a_{kl}(Aw)_kw_l$ we get $f(\xi)=2\log\|A\|$. This
of course does not imply that the supremum of $H_\varphi(M|N)$ over
all states $\varphi$ equals $2\log\|A\|$, even if $n_k\ge a_{kl}$,
since we only know that our upper bound on $H_\varphi(M|N)$ is sharp
for tracial states, and for tracial states the numbers
$\varphi(z_kw_l)$ cannot be
arbitrary.

\bigskip

\section{Canonical shift on the tower of relative commutants}

Let $N\subset M$ be a finite index inclusion of II$_1$ factors. Put
$M_{-1}=N$ and $M_0=M$. Iterating the basic extension with respect to
the trace-preserving conditional expectations we get the Jones tower
$$
M_{-1}\subset M_0\subset M_1\subset M_2\subset\cdots.
$$
We also choose a tunnel
$$
\cdots\subset M_{-3}\subset M_{-2}\subset M_{-1}\subset M_0.
$$
Let $e_{j}\in M_{j+1}$, $j\in\Z$, be the corresponding Jones
projections. Denote by $\tau$ the unique tracial state on $\bigcup_n M_n$
and by $E_j$ the $\tau$-preserving conditional expectation
$\bigcup_nM_n\to M_j$. Thus, $M_{j+1}=M_je_jM_j$, meaning that $M_{j+1}$
is spanned by the elements $xe_jy$ for $x,y\in M_j$,
$e_jxe_j=E_{j-1}(x)e_j$ for $x\in M_j$, and $E_j(e_j)=\lambda1$, where
$\lambda=[M : N]^{-1}$.

\smallskip

Consider the canonical shift $\gamma$ on $\bigcup_{j<k}(M_j'\cap
M_k)$. This is an automorphism such that $\gamma(e_j)=e_{j+2}$ and
$\gamma(M_j'\cap M_k)=M_{j+2}'\cap M_{k+2}$. It can be defined as
follows, see,~e.g.,~\cite{MR2251116}*{Section~10.4}. The representation
of $M_{j+1}$ on $L^2(M_j)$ given by the definition of the basic
extension extends uniquely to a representation of $\bigcup_nM_n$ such
that
$$
e_{j+n}\Lambda_j(x)=J_je_{j-n}J_j\Lambda_j(x)=\Lambda_j(xe_{j-n})\ \
\text{for all}\ n\ge1,
$$
where $\Lambda_j\colon M_j\to L^2(M_j)$ is the GNS-map and $J_j$ is
the modular conjugation. In this representation we have
$M_{j+n}=J_jM_{j-n}'J_j$. Define a $*$-anti-automorphism $\gamma_j$ of
$\bigcup_{i<k}(M_i'\cap M_k)$ by
$$
\gamma_j(x)=J_jx^*J_j\ \ \text{on}\ L^2(M_j).
$$
The canonical shift is defined by $\gamma=\gamma_{j+1}\gamma_j$. This
definition is independent of $j\in\Z$. The automorphism $\gamma$ is
completely characterized by the properties that it maps $M_j'\cap M_k$
into $M_{j+2}'\cap M_{k+2}$ and satisfies
\begin{equation}\label{eshift}
\gamma(x)e_{j+1}=\lambda^{j-k}e_{j+1}\dots e_k x e_{k+1}e_k\dots
e_{j+1}\ \ \text{for}\ \ x\in M_j'\cap M_k,\ j<k.
\end{equation}

The following proposition, which is at the origin of the
classification theory of subfactors, is a well-known result of
Ocneanu.

\begin{proposition} \label{pocneanu} There is an isomorphism of the
inductive system
$$
\Endd_\NN(L^2(M))\xrightarrow{\iota_{L^2(M)}\otimes\,\cdot}\Endd_\NN(L^2(M)^{\otimes_N
2})\xrightarrow{\iota_{L^2(M)}\otimes\,\cdot} \Endd_\NN(L^2(M)^{\otimes_N
3})\to \dots
$$
onto the system
$$
M_{-1}'\cap M_1\to M'_{-3}\cap M_1\to M_{-5}'\cap M_1\to \dots\ ,
$$
where all the arrows are the inclusion maps, such that the shift
endomorphism $T\mapsto T\otimes\iota_{L^2(M)}$ of
$\bigcup_{n\ge1}\Endd_\NN(L^2(M)^{\otimes_N n})$ corresponds to the
endomorphism $\gamma^{-1}$ of $\bigcup_{n\ge1}(M_{-2n+1}'\cap M_1)$.
\end{proposition}

Despite being well-known, this is usually formulated in a weaker form,
see, e.g.,~\cite{MR1424954}*{Section~4}, and it seems to be difficult to
find a clear complete proof of the proposition as it is stated above
in the literature. We will therefore sketch a possible proof for the
reader's convenience.

\smallskip

We start by considering the $N$-bimodule maps
$$
u_{m,n}\colon L^2(M_m)\otimes_N L^2(M_n)\to L^2(M_{m+n+1})\ \
(m,n\ge0),
$$
\begin{align*}
u_{m,n}(\Lambda_m(x)\otimes\Lambda_n(y))&=\lambda^{-(m+1)(n+1)/2}\Lambda_{m+n+1}(xe_m\dots
e_0e_{m+1}\dots e_1\dots e_{m+n}\dots e_ny)\\
&=\lambda^{-(m+1)(n+1)/2}\Lambda_{m+n+1}(xe_m\dots e_{m+n}
e_{m-1}\dots e_{m+n-1}\dots e_0\dots e_ny).
\end{align*}

\begin{lemma} The maps $u_{m,n}$ are unitary, and the following
identities hold:
\begin{equation} \label{eassoc}
u_{k+m+1,n}(u_{k,m}\otimes\iota_{L^2(M_n)}) =
u_{k,m+n+1}(\iota_{L^2(M_k)}\otimes u_{m,n}).
\end{equation}
\end{lemma}

\bp Using the identities
$$
E_{m+n-k}(e_{m+n-k}\dots e_{m-k}E_{m-k}(x^*x)e_{m-k}\dots e_{m+n-k})
=\lambda^{n+1}E_{m-k-1}(x^*x)
$$
for $k=0,\dots,m$, it is easy to check that the maps $u_{m,n}$ are
isometric. Identity \eqref{eassoc} is also straightforward. To prove
surjectivity of $u_{m,n}$, observe first that
$$
M_me_m\dots e_{m+n}M_{m+n}=M_{m+n+1}.
$$
This can be seen by induction on $n$, using that
$$
M_me_m\dots e_{m+n}M_{m+n}= M_me_m\dots
e_{m+n-1}M_{m+n-1}e_{m+n}M_{m+n}
$$
and $M_{m+n}e_{m+n}M_{m+n}=M_{m+n+1}$. From this, in turn, by
induction on $m$ we get
$$
M_me_m\dots e_{m+n} e_{m-1}\dots e_{m+n-1}\dots e_0\dots
e_nM_n=M_{m+n+1},
$$
since the left hand side can be written as
$$
M_me_m\dots e_{m+n}M_{m-1}e_{m-1}\dots e_{m+n-1}\dots e_0\dots e_nM_n.
$$
This proves surjectivity of $u_{m,n}$.  \ep

By distributing parentheses in $L^2(M)^{\otimes_N n}$, e.g.,~as
$$
L^2(M)\otimes_N(L^2(M)\otimes_N(\dots(L^2(M)\otimes_N L^2(M))\dots)),
$$
and using the isomorphisms $u_{k,m}$, we get an isomorphism $v_n\colon
L^2(M)^{\otimes_N n}\to L^2(M_{n-1})$, and hence an isomorphism
$$
\psi_n\colon\Endd_\NN(L^2(M)^{\otimes_N n})\to N'\cap M_{2n-1}.
$$
By \eqref{eassoc} the isomorphism $v_n$, and hence $\psi_n$, is
independent of the way we distribute parentheses in~$L^2(M)^{\otimes_N
n}$.

\begin{lemma} \label{lshift2} For any
$T\in \Endd_\NN(L^2(M)^{\otimes_N n})$ we have
$$
\psi_{n+1}(T\otimes\iota_{L^2(M)})=\psi_n(T)\ \ \text{and}\ \
\psi_{n+1}(\iota_{L^2(M)}\otimes T)=\gamma(\psi_n(T)).
$$
\end{lemma}

\bp As $v_{n+1}=u_{n-1,0}(v_n\otimes\iota)$, for the first equality it
suffices to show that
$$
u_{n-1,0}\colon L^2(M_{n-1})\otimes_N L^2(M)\to L^2(M_n)
$$
is a left $M_{2n-1}$-module map. It is clear that $u_{n-1,0}$ is an
$M_{n-1}$-module map. Therefore it is enough to show that
$u_{n-1,0}e_k=e_ku_{n-1,0}$ for $k=n-1,\dots,2n-2$. Consider three
cases.

\smallskip

(i) $k=n-1$. We have, for $x\in M_{n-1}$ and $y\in M$, that
$$
u_{n-1,0}(e_{n-1}\Lambda_{n-1}(x)\otimes\Lambda(y))
=\lambda^{-n/2}\Lambda_n(E_{n-2}(x)e_{n-1}\dots e_0y)
=\lambda^{-n/2}\Lambda_n(e_{n-1}xe_{n-1}\dots e_0y),
$$
which is what we need.

\smallskip

(ii) $k=n$. In this case, using that
$e_n\Lambda_{n-1}(x)=\Lambda_{n-1}(xe_{n-2})$, we get
$$
u_{n-1,0}(e_{n}\Lambda_{n-1}(x)\otimes\Lambda(y))
=\lambda^{-n/2}\Lambda_n(xe_{n-2}e_{n-1}\dots e_0y)
=\lambda^{-n/2+1}\Lambda_n(xe_{n-2}\dots e_0y).
$$
On the other hand,
$$
e_nu_{n-1,0}(\Lambda_{n-1}(x)\otimes\Lambda(y))
=\lambda^{-n/2}\Lambda_n(E_{n-1}(xe_{n-1}\dots e_0y))
=\lambda^{-n/2+1}\Lambda_n(xe_{n-2}\dots e_0y),
$$
so again $u_{n-1,0}e_n=e_nu_{n-1,0}$.

\smallskip

(iii) $n<k\le 2n-2$. Using again that $e_k=J_{n-1}e_{2n-k-2}J_{n-1}$
on $L^2(M_{n-1})$ and $e_k=J_{n}e_{2n-k}J_{n}$ on~$L^2(M_{n})$, we get
$$
u_{n-1,0}(e_k\Lambda_{n-1}(x)\otimes\Lambda(y))
=\lambda^{-n/2}\Lambda_n(xe_{2n-k-2}e_{n-1}\dots e_0y)
$$
and
$$
e_ku_{n-1,0}(\Lambda_{n-1}(x)\otimes\Lambda(y))
=\lambda^{-n/2}\Lambda_n(xe_{n-1}\dots e_0ye_{2n-k}).
$$
As $e_{2n-k}$ commutes with $y\in M$ and
$$
e_{2n-k-2}e_{n-1}\dots e_0=\lambda e_{n-1}\dots e_{2n-k}
e_{2n-k-2}\dots e_0=e_{n-1}\dots e_0e_{2n-k},
$$
this gives $u_{n-1,0}e_k=e_ku_{n-1,0}$.

\smallskip

Turning to the second identity, as $v_{n+1}=u_{0,n-1}(\iota\otimes
v_n)$, it suffices to show that
$$
u_{0,n-1}(\iota\otimes x)=\gamma(x)u_{0,n-1}\ \ \text{for}\ \ x\in
N'\cap M_{2n-1}.
$$
Since $\gamma(x)\in M_1'\cap M_{2n+1}$ commutes with $M$, recalling
the definition of $u_{0,n-1}$ we see that this boils down to showing
that $w_{n-1}x=\gamma(x)w_{n-1}$, where $w_{n-1}\colon L^2(M_{n-1})\to
L^2(M_n)$ is defined by
$$
w_{n-1}\Lambda_{n-1}(y)=\Lambda_n(e_0\dots e_{n-1}y).
$$
It is convenient to prove a stronger statement. Define a map $\pi\colon M_{2n-1}\to M_{2n+1}$ by
$$
\pi(x)=\lambda^{-2n}e_{0}\dots e_{2n-1} x e_{2n}e_{2n-1}\dots e_{0}.
$$
It is easy to see that $\pi$ is a $*$-homomorphism. By \eqref{eshift}
we also have $\pi(x)=\gamma(x)e_0$ for $x\in N'\cap M_{2n-1}$. It
follows that in order to prove the second part of the lemma it
suffices to show that
$$
w_{n-1}x=\pi(x)w_{n-1}\ \ \text{for all}\ \ x\in M_{2n-1}.
$$
Since $\pi(x)=xe_0$ for $x\in N$, this identity holds for $x\in
N$. Hence to finish the proof it is enough to check this identity for
$x=e_k$, $k=-1,\dots,2n-2$. That is, we have to show that
$w_{n-1}e_{-1}=\lambda^{-1}e_0e_{-1}e_1e_0w_{n-1}$ and
$w_{n-1}e_k=e_{k+2}e_0w_{n-1}$ for $k=0,\dots,2n-2$. This is done
similarly to the first part of the proof of the lemma by considering
different cases: two cases for
$w_{n-1}e_{-1}=\lambda^{-1}e_0e_{-1}e_1e_0w_{n-1}$ corresponding to
$n=1$ and $n>1$, and four cases for $w_{n-1}e_k=e_{k+2}e_0w_{n-1}$
corresponding to $0\le k< n-2$, $k=n-2$ ($n\ge2$), $k=n-1$ ($n\ge1$)
and $n\le k\le 2n-2$.  \ep

\bp[Proof of Proposition~\ref{pocneanu}] It follows from
Lemma~\ref{lshift2} that the isomorphisms
$$
\gamma^{-(n-1)}\psi_n\colon\Endd_\NN(L^2(M)^{\otimes_N n})\to
M_{-2n+1}'\cap M_1
$$
define the required isomorphism of the inductive systems.
\ep

\bigskip

\raggedright

\bigskip

\end{document}